\documentclass[12pt,a4paper]{article}
\usepackage{indentfirst}
\usepackage{amsfonts}
\usepackage{amssymb}
\usepackage{mathrsfs}
\usepackage{amsmath}
\usepackage{amsthm}
\usepackage{enumerate}
\usepackage{cite}
\usepackage{mathrsfs}
\allowdisplaybreaks
\newtheorem{defn}{Definition}[section]
\newtheorem{thm}[defn]{Theorem}
\newtheorem{lem}[defn]{Lemma}
\newtheorem{prop}[defn]{Proposition}

\newtheorem{ex}[defn]{Example}
\newtheorem{re}[defn]{Remark}
\begin{document}
\title{{\bf Generalized derivations, quasiderivations and centroids of ternary Jordan algebras}}
\author{\normalsize \bf Chenrui Yao, Yao Ma,  Liangyun Chen}
\date{{\small{ School of Mathematics and Statistics,  Northeast Normal University,\\ Changchun 130024, CHINA
}}} \maketitle
\date{}

   {\bf\begin{center}{Abstract}\end{center}}

In this paper, we give some construction about ternary Jordan algebras at first. Next we study relationships between generalized derivations, quasiderivations and centroids of ternary Jordan algebras. We show that for ternary Jordan algebras, generalized derivation algebras are the sum of quasiderivation algebras and centroids where centroids are ideals of generalized derivation algebras. We also prove that quasiderivations can be embedded into larger ternary Jordan algebras as derivations. In particular, we also determine dimensions of ternary Jordan algebras in the case of all linear transformations are quasiderivations. Some properties about centroids of ternary Jordan algebras are also displayed.

\renewcommand{\thefootnote}{\fnsymbol{footnote}}
\footnote[0]{ Corresponding author(L. Chen): chenly640@nenu.edu.cn.}
\footnote[0]{Supported by  NNSF of China (Nos. 11771069 and 11801066), NSF of  Jilin province (No. 20170101048JC),   the project of Jilin province department of education (No. JJKH20180005K) and the Fundamental Research Funds for the Central Universities(No. 130014801).}

\section{Introduction}
In 1932, the physicist Jordan proposed a program to discover a new algebraic setting for quantum mechanics, so Jordan algebras were created in this proceeding. Moreover, Jordan algebras were truned out to have illuminating connections with many areas of mathematics. Albert renamed them Jordan algebras and developed a successful structure theory for Jordan algebras over any field of characteristic zero in \cite{A1}.

As a generalization of Jordan algebras, Kaygorodov, Pozhidaev and Saraiva gave the definition of $n$-ary Jordan algebras in \cite{IAP}. They constructed some examples of ternary Jordan algebras as well, including construction via matrix algebras and Cayley-Dickson algebras. What's more, they also constructed a class of ternary Jordan algebras by TKK-construction, which is famous in the classical Jordan algebra theory. In this paper, we also give some construction via Jordan algebras and associative algebras. What's more, we also give a kind of construction about Jordan algebras via ternary Jordan algebras. In a sense, Jordan algebras are a special case of ternary Jordan algebras. So in the following, let's turn our attention to ternary Jordan algebras.

As is well known, derivations and generalized derivations play an important role in the research of structure and property in algebraic system. In \cite{HLS}, authors developed an approach to deformations of the Witt and Virasoro algebras based on $\sigma$-derivations. In \cite{LS}, authors showed in \cite{HLS} that when the deformation scheme is applied to $sl_{2}(\rm{F})$ one can, by choosing parameters suitably, deform $sl_{2}(\rm{F})$ into the Heisenberg Lie algebra and some other 3-dimensional Lie algebras in addition to more exotic types of algebras, this being in stark contrast to the classical deformation schemes where $sl_{2}(\rm{F})$ is rigid. Study about derivations originated from \cite{F1}. Bre$s$ar has made great contributions to the research of derivations. In \cite{B1}, he determined types of derivations on $R \otimes S$ where $R$ and $S$ are two nonassociative unital algebras under some assumptions. More results about derivations refer to \cite{B2, BCS}. The research on generalized derivations was started by Leger and Luks in \cite{L2}. They gave many properties about generalized derivations of Lie algebras. From then on, many authors generalized these results to other algebras in \cite{Z1,LW,CMN,ZCM,S1}.

And in this paper, we focus on ternary Jordan algebras and develop a complete theory of generalized derivations of ternary Jordan algebras. It's worth noting that quasicentroid of a ternary Jordan algebra $\mathcal{A}$ is a subalgebra of the general linear Lie algebra $gl(\mathcal{A})$, which is different from the case of Jordan algebras and Jordan superalgebras. As an application, quasicentroids are also ideals of the generalized derivation algebras.

We proceed as follows. In Section \ref{se:2}, we mainly give some basic definitions which will be used in the following. In section \ref{se:3}, we give some construction about ternary Jordan algebras. In Section \ref{se:4}, we develop some elementary properties about generalized derivations, quasiderivations, centroids and quasicentroids of ternary Jordan alegbras, some of which are technical results to be used in the sequel. Section \ref{se:5} is devoted to show that quasiderivations of a ternary Jordan algebra can be embedded as derivations into a lager ternary Jordan algebra. And we obtain a direct sum decomposition of the derivation algebra of the lager ternary Jordan algebra under some conditions. In Section \ref{se:6}, we characterize a class of ternary Jordan algebras satisfying all linear transformations are quasiderivations. We get their dimensions are not more than $2$. In particular, we obtain the necessary conditions for the corresponding dimensions. In Section \ref{se:7}, we turn to centroids of ternary Jordan algebras and prove some properties about them.
\section{Preliminaries}\label{se:2}
\begin{defn}\cite{M1}
An algebra $J$ over a field $\rm{F}$ is a Jordan algebra satisfying for any $x, y \in J$,
\begin{enumerate}[(1)]
\item $x \circ y = y \circ x$;
\item $(x^{2} \circ y) \circ x = x^{2} \circ (y \circ x)$.
\end{enumerate}
\end{defn}
\begin{defn}\cite{IAP}
Let $\mathcal{A}$ be an $n$-ary algebra with a multilinear multiplication $\rm{[\![\cdot,\cdots,\cdot]\!]} : \times^{n}\mathbb{V} \rightarrow \mathbb{V}$, where $\mathbb{V}$ is the underlying vector space. $\mathcal{A}$ is said to be an $n$-ary Jordan algebra if for every $x_{1}, x_{2}, \cdots, x_{n}, y_{2}, \cdots, y_{n} \in \mathbb{V}$,
\begin{enumerate}[(1)]
\item $\rm{[\![\it{x_{\sigma(1)}},\cdots,\it{x_{\sigma(n)}}]\!]} = \rm{[\![\it{x_{1}},\cdots,\it{x_{n}}]\!]}$ for every permutation $\sigma \in \mathcal{S}_{n}$;
\item $[R_{(x_{2}, \cdots, x_{n})}, R_{(y_{2}, \cdots, y_{n})}] \in Der(\mathcal{A})$ where $[\cdot,\cdot]$ stands again for the commutator and $R_{(x_{2}, \cdots, x_{n})}$, $R_{(y_{2}, \cdots, y_{n})}$ are the right multiplication operators, defined in the usual way:
    \[y \mapsto R_{(y_{2}, \cdots, y_{n})}(y) = \rm{[\![\it{y},\it{y_{2}},\cdots,\it{y_{n}}]\!]}.\]
\end{enumerate}
\end{defn}
\begin{defn}
Let $\mathcal{A}$ be a ternary Jordan algebra. A subspace $\mathcal{I}$ of $\mathcal{A}$ is called an ideal of $\mathcal{A}$ if $\rm{[\![\mathcal{I},\mathcal{A},\mathcal{A}]\!]} \subseteq \mathcal{I}$.
\end{defn}
\begin{defn}
Let $\mathcal{A}$ be a ternary Jordan algebra. We define the set $\{z \in \mathcal{A} \mid \rm{[\![\it{x},\it{y},\it{z}]\!]} = 0,\;\forall \it{x, y} \in \mathcal{A}\}$ to be the annihilator of $\mathcal{A}$, denoted by $Z(\mathcal{A})$. It's obvious that $Z(\mathcal{A})$ is an ideal of $\mathcal{A}$.
\end{defn}
\begin{defn}
Suppose that $\mathcal{A}$ is a ternary Jordan algebra over $\rm{F}$ and $f : \mathcal{A} \times \mathcal{A} \rightarrow \rm{F}$ is a bilinear form on $\mathcal{A}$. If for any $x, y, z , w \in \mathcal{A}$, $f$ satisifies
\[f(\rm{[\![\it{x},\it{y},\it{z}]\!]}, \it{w}) = f(\it{x}, \rm{[\![\it{w},\it{y},\it{z}]\!]}),\]
then $f$ is called invariant.
\end{defn}
\begin{defn}
Let $\mathcal{A}$ be a ternary Jordan algebra, $\it{f}$ be a linear map of $\mathcal{A}$. If $\it{f}$ satisfies $\it{f}(\mathcal{A}) \subseteq Z(\mathcal{A})$ and $\it{f}(\rm{[\![\mathcal{A},\mathcal{A},\mathcal{A}]\!]}) = 0$, then $\it{f}$ is called a center derivation of $\mathcal{A}$. The set of all center derivations is denoted by $ZDer(\mathcal{A})$, which is an ideal of $Der(\mathcal{A})$.
\end{defn}
\begin{defn}
Let $\mathcal{A}$ be a ternary Jordan algebra and $\it{f}_{1}$ a linear map on $\mathcal{A}$. If there exist linear maps $\it{f}_{2}$, $\it{f}_{3}$, $\it{f}^{'}$ on $\mathcal{A}$ satisfy for all $x, y, z \in \mathcal{A}$,
\[\rm{[\![\it{f}_{1}(x),\it{y},\it{z}]\!]} + \rm{[\![\it{x},\it{f}_{2}(y),\it{z}]\!]} + \rm{[\![\it{x},\it{y},\it{f}_{3}(z)]\!]} = \it{f}^{'}(\rm{[\![\it{x},\it{y},\it{z}]\!]}),\]
then $\it{f}_{1}$ is called a generalized derivation of $\mathcal{A}$. The set of all quaternions $(\it{f}_{1},\it{f}_{2},\it{f}_{3},\it{f}^{'})$ is denoted by $\Delta(\mathcal{A})$, and the set of all generalized derivations is denoted by $GDer(\mathcal{A})$.
\end{defn}
\begin{re}
\begin{enumerate}[(1)]
\item According to the symmetry of the multiplication, if $(\it{f}_{1},\it{f}_{2},\it{f}_{3},\it{f}^{'}) \in \Delta(\mathcal{A})$, then
\[\rm{[\![\it{f}_{2}(x),\it{y},\it{z}]\!]} + \rm{[\![\it{x},\it{f}_{1}(y),\it{z}]\!]} + \rm{[\![\it{x},\it{y},\it{f}_{3}(z)]\!]} = \it{f}^{'}(\rm{[\![\it{x},\it{y},\it{z}]\!]}).\]
It follows that $\it{f}_{1}, \it{f}_{2}, \it{f}_{3} \in GDer(\mathcal{A})$ and $(\it{f}_{i_{1}},\it{f}_{i_{2}},\it{f}_{i_{3}},\it{f}^{'}) \in \Delta(\mathcal{A})$, where $(i_{1},i_{2},i_{3})$ is any $3$-ary permutation.
\item In the case of $\it{f}_{1} = \it{f}_{2} = \it{f}_{3} = \it{f}$, the linear map $\it{f}$ is called a quasiderivation of $\mathcal{A}$, and the set of all quasiderivations is denoted by $QDer(\mathcal{A})$.
\item For any ternary Jordan algebra $\mathcal{A}$, we have $ZDer(\mathcal{A}) \subseteq Der(\mathcal{A}) \subseteq QDer(\mathcal{A}) \subseteq GDer(\mathcal{A})$.
\end{enumerate}
\end{re}
\begin{defn}
Let $\mathcal{A}$ be a ternary Jordan algebra. The centroid of $\mathcal{A}$ is a vector space spanned by all elements $\it{f} \in \rm{End}(\mathcal{A})$ which satisfies for all $x, y, z \in \mathcal{A}$
\[\rm{[\![\it{f}(x),\it{y},\it{z}]\!]} = \rm{[\![\it{x},\it{f}(y),\it{z}]\!]} = \rm{[\![\it{x},\it{y},\it{f}(z)]\!]} = \it{f}(\rm{[\![\it{x},\it{y},\it{z}]\!]}),\]
denoted by $\Gamma(\mathcal{A})$.
\end{defn}
\begin{defn}
Let $\mathcal{A}$ be a ternary Jordan algebra. The quasicentroid of $\mathcal{A}$ is a vector space spanned by all elements $\it{f} \in \rm{End}(\mathcal{A})$ which satisfies for all $x, y, z \in \mathcal{A}$
\[\rm{[\![\it{f}(x),\it{y},\it{z}]\!]} = \rm{[\![\it{x},\it{f}(y),\it{z}]\!]} = \rm{[\![\it{x},\it{y},\it{f}(z)]\!]},\]
denoted by $Q\Gamma(\mathcal{A})$.
\end{defn}

Obviously, $\Gamma(\mathcal{A}) \subseteq Q\Gamma(\mathcal{A})$.
\section{Some construction of ternary Jordan algebras}\label{se:3}
\begin{prop}\label{prop:3.1}
Suppose that $J$ is a Jordan algebra over $\rm{F}$ and $\alpha \in J^{*}$ satisfies $\alpha(J \circ J) = 0$. Define $\rm{[\![\cdot,\cdot,\cdot]\!]} : \it{J \times J \times J \rightarrow J}$ by
\[\rm{[\![\it{x},\it{y},\it{z}]\!]} = \it{\alpha(x)y \circ z} + \it{\alpha(y)z \circ x} + \it{\alpha(z)x \circ y},\quad\forall x, y, z \in J.\]
Then $(J, \rm{[\![\cdot,\cdot,\cdot]\!]})$ is a ternary Jordan algebra, denoted by $J_{\alpha}$.
\end{prop}
\begin{proof}
For any $u, v, w, x_{1}, x_{2}, y_{1}, y_{2} \in J$, we have
\begin{align*}
&D_{(x_{1}, x_{2}), (y_{1}, y_{2})}(\rm{[\![\it{u},\it{v},\it{w}]\!]})\\
&= \alpha(w)\alpha(x_{1})\alpha(y_{1})((u \circ v) \circ x_{2}) \circ y_{2} + \alpha(w)\alpha(x_{1})\alpha(y_{2})((u \circ v) \circ x_{2}) \circ y_{1}\\
&+ \alpha(w)\alpha(x_{2})\alpha(y_{1})((u \circ v) \circ x_{1}) \circ y_{2} + \alpha(w)\alpha(x_{2})\alpha(y_{2})((u \circ v) \circ x_{1}) \circ y_{1}\\
&+ \alpha(v)\alpha(x_{1})\alpha(y_{1})((w \circ u) \circ x_{2}) \circ y_{2} + \alpha(v)\alpha(x_{1})\alpha(y_{2})((w \circ u) \circ x_{2}) \circ y_{1}\\
&+ \alpha(v)\alpha(x_{2})\alpha(y_{1})((w \circ u) \circ x_{1}) \circ y_{2} + \alpha(v)\alpha(x_{2})\alpha(y_{2})((w \circ u) \circ x_{1}) \circ y_{1}\\
&+ \alpha(u)\alpha(x_{1})\alpha(y_{1})((v \circ w) \circ x_{2}) \circ y_{2} + \alpha(u)\alpha(x_{1})\alpha(y_{2})((v \circ w) \circ x_{2}) \circ y_{1}\\
&+ \alpha(u)\alpha(x_{2})\alpha(y_{1})((v \circ w) \circ x_{1}) \circ y_{2} + \alpha(u)\alpha(x_{2})\alpha(y_{2})((v \circ w) \circ x_{1}) \circ y_{1}\\
&- \alpha(w)\alpha(y_{1})\alpha(x_{1})((u \circ v) \circ y_{2}) \circ x_{2} - \alpha(w)\alpha(y_{1})\alpha(x_{2})((u \circ v) \circ y_{2}) \circ x_{1}\\
&- \alpha(w)\alpha(y_{2})\alpha(x_{1})((u \circ v) \circ y_{1}) \circ x_{2} - \alpha(w)\alpha(y_{2})\alpha(x_{2})((u \circ v) \circ y_{1}) \circ x_{1}\\
&- \alpha(v)\alpha(y_{1})\alpha(x_{1})((w \circ u) \circ y_{2}) \circ x_{2} - \alpha(v)\alpha(y_{1})\alpha(x_{2})((w \circ u) \circ y_{2}) \circ x_{1}\\
&- \alpha(v)\alpha(y_{2})\alpha(x_{1})((w \circ u) \circ y_{1}) \circ x_{2} - \alpha(v)\alpha(y_{2})\alpha(x_{2})((w \circ u) \circ y_{1}) \circ x_{1}\\
&- \alpha(u)\alpha(y_{1})\alpha(x_{1})((v \circ w) \circ y_{2}) \circ x_{2} - \alpha(u)\alpha(y_{1})\alpha(x_{2})((v \circ w) \circ y_{2}) \circ x_{1}\\
&- \alpha(u)\alpha(y_{2})\alpha(x_{1})((v \circ w) \circ y_{1}) \circ x_{2} - \alpha(u)\alpha(y_{2})\alpha(x_{2})((v \circ w) \circ y_{1}) \circ x_{1},
\end{align*}
on the other hand,
\begin{align*}
&\rm{[\![\it{D_{(x_{1}, x_{2}), (y_{1}, y_{2})}(u)},\it{v},\it{w}]\!]} + \rm{[\![\it{u},\it{D_{(x_{1}, x_{2}), (y_{1}, y_{2})}(v)},\it{w}]\!]} + \rm{[\![\it{u},\it{v},\it{D_{(x_{1}, x_{2}), (y_{1}, y_{2})}(w)}]\!]}\\
&= \alpha(x_{1})\alpha(y_{1})\alpha(w)((u \circ x_{2}) \circ y_{2}) \circ v + \alpha(x_{1})\alpha(y_{1})\alpha(v)((u \circ x_{2}) \circ y_{2}) \circ w\\
&+ \alpha(x_{1})\alpha(y_{2})\alpha(w)((u \circ x_{2}) \circ y_{1}) \circ v + \alpha(x_{1})\alpha(y_{2})\alpha(v)((u \circ x_{2}) \circ y_{1}) \circ w\\
&+ \alpha(x_{2})\alpha(y_{1})\alpha(w)((u \circ x_{1}) \circ y_{2}) \circ v + \alpha(x_{2})\alpha(y_{1})\alpha(v)((u \circ x_{1}) \circ y_{2}) \circ w\\
&+ \alpha(x_{2})\alpha(y_{2})\alpha(w)((u \circ x_{1}) \circ y_{1}) \circ v + \alpha(x_{2})\alpha(y_{2})\alpha(v)((u \circ x_{1}) \circ y_{1}) \circ w\\
&- \alpha(y_{1})\alpha(x_{1})\alpha(w)((u \circ y_{2}) \circ x_{2}) \circ v - \alpha(y_{1})\alpha(x_{1})\alpha(v)((u \circ y_{2}) \circ x_{2}) \circ w\\
&- \alpha(y_{1})\alpha(x_{2})\alpha(w)((u \circ y_{2}) \circ x_{1}) \circ v - \alpha(y_{1})\alpha(x_{2})\alpha(v)((u \circ y_{2}) \circ x_{1}) \circ w\\
&- \alpha(y_{2})\alpha(x_{1})\alpha(w)((u \circ y_{1}) \circ x_{2}) \circ v - \alpha(y_{2})\alpha(x_{1})\alpha(v)((u \circ y_{1}) \circ x_{2}) \circ w\\
&- \alpha(y_{2})\alpha(x_{2})\alpha(w)((u \circ y_{1}) \circ x_{1}) \circ v - \alpha(y_{2})\alpha(x_{2})\alpha(v)((u \circ y_{1}) \circ x_{1}) \circ w\\
&+ \alpha(x_{1})\alpha(y_{1})\alpha(u)((v \circ x_{2}) \circ y_{2}) \circ w + \alpha(x_{1})\alpha(y_{1})\alpha(w)((v \circ x_{2}) \circ y_{2}) \circ u\\
&+ \alpha(x_{1})\alpha(y_{2})\alpha(u)((v \circ x_{2}) \circ y_{1}) \circ w + \alpha(x_{1})\alpha(y_{2})\alpha(w)((v \circ x_{2}) \circ y_{1}) \circ u\\
&+ \alpha(x_{2})\alpha(y_{1})\alpha(u)((v \circ x_{1}) \circ y_{2}) \circ w + \alpha(x_{2})\alpha(y_{1})\alpha(w)((v \circ x_{1}) \circ y_{2}) \circ u\\
&+ \alpha(x_{2})\alpha(y_{2})\alpha(u)((v \circ x_{1}) \circ y_{1}) \circ w + \alpha(x_{2})\alpha(y_{2})\alpha(w)((v \circ x_{1}) \circ y_{1}) \circ u\\
&- \alpha(y_{1})\alpha(x_{1})\alpha(u)((v \circ y_{2}) \circ x_{2}) \circ w - \alpha(y_{1})\alpha(x_{1})\alpha(w)((v \circ y_{2}) \circ x_{2}) \circ u\\
&- \alpha(y_{1})\alpha(x_{2})\alpha(u)((v \circ y_{2}) \circ x_{1}) \circ w - \alpha(y_{1})\alpha(x_{2})\alpha(w)((v \circ y_{2}) \circ x_{1}) \circ u\\
&- \alpha(y_{2})\alpha(x_{1})\alpha(u)((v \circ y_{1}) \circ x_{2}) \circ w - \alpha(y_{2})\alpha(x_{1})\alpha(w)((v \circ y_{1}) \circ x_{2}) \circ u\\
&- \alpha(y_{2})\alpha(x_{2})\alpha(u)((v \circ y_{1}) \circ x_{1}) \circ w - \alpha(y_{2})\alpha(x_{2})\alpha(w)((v \circ y_{1}) \circ x_{1}) \circ u\\
&+ \alpha(x_{1})\alpha(y_{1})\alpha(u)((w \circ x_{2}) \circ y_{2}) \circ v + \alpha(x_{1})\alpha(y_{1})\alpha(v)((w \circ x_{2}) \circ y_{2}) \circ u\\
&+ \alpha(x_{1})\alpha(y_{2})\alpha(u)((w \circ x_{2}) \circ y_{1}) \circ v + \alpha(x_{1})\alpha(y_{2})\alpha(v)((w \circ x_{2}) \circ y_{1}) \circ u\\
&+ \alpha(x_{2})\alpha(y_{1})\alpha(u)((w \circ x_{1}) \circ y_{2}) \circ v + \alpha(x_{2})\alpha(y_{1})\alpha(v)((w \circ x_{1}) \circ y_{2}) \circ u\\
&+ \alpha(x_{2})\alpha(y_{2})\alpha(u)((w \circ x_{1}) \circ y_{1}) \circ v + \alpha(x_{2})\alpha(y_{2})\alpha(v)((w \circ x_{1}) \circ y_{1}) \circ u\\
&- \alpha(y_{1})\alpha(x_{1})\alpha(u)((w \circ y_{2}) \circ x_{2}) \circ v - \alpha(y_{1})\alpha(x_{1})\alpha(v)((w \circ y_{2}) \circ x_{2}) \circ u\\
&- \alpha(y_{1})\alpha(x_{2})\alpha(u)((w \circ y_{2}) \circ x_{1}) \circ v - \alpha(y_{1})\alpha(x_{2})\alpha(v)((w \circ y_{2}) \circ x_{1}) \circ u\\
&- \alpha(y_{2})\alpha(x_{1})\alpha(u)((w \circ y_{1}) \circ x_{2}) \circ v - \alpha(y_{2})\alpha(x_{1})\alpha(v)((w \circ y_{1}) \circ x_{2}) \circ u\\
&- \alpha(y_{2})\alpha(x_{2})\alpha(u)((w \circ y_{1}) \circ x_{1}) \circ v - \alpha(y_{2})\alpha(x_{2})\alpha(v)((w \circ y_{1}) \circ x_{1}) \circ u.
\end{align*}

Since $J$ is a Jordan algebra, we have $[R_{x}, R_{y}] \in Der(J)$, that is to say for any $u, v \in J$
\begin{equation}\label{eq:3.1}
R_{x}R_{y}(u \circ v) - R_{y}R_{x}(u \circ v) = R_{x}R_{y}(u) \circ v - R_{y}R_{x}(u) \circ v + R_{x}R_{y}(v) \circ u - R_{y}R_{x}(v) \circ u.\tag{3.1}
\end{equation}

Using (\ref{eq:3.1}), we have
\begin{align*}
&D_{(x_{1}, x_{2}), (y_{1}, y_{2})}(\rm{[\![\it{u},\it{v},\it{w}]\!]}) = \rm{[\![\it{D_{(x_{1}, x_{2}), (y_{1}, y_{2})}(u)},\it{v},\it{w}]\!]} + \rm{[\![\it{u},\it{D_{(x_{1}, x_{2}), (y_{1}, y_{2})}(v)},\it{w}]\!]}\\
&+ \rm{[\![\it{u},\it{v},\it{D_{(x_{1}, x_{2}), (y_{1}, y_{2})}(w)}]\!]},
\end{align*}
which implies that $D_{(x_{1}, x_{2}), (y_{1}, y_{2})} \in Der(J_{\alpha})$. Therefore, $(J, \rm{[\![\cdot,\cdot,\cdot]\!]})$ is a ternary Jordan algebra.
\end{proof}
\begin{prop}\label{prop:3.2}
Suppose that $\mathcal{A}$ is a ternary Jordan algebra and $z_{0} \in \mathcal{A}$ satisfies $R_{(z_{0}, z_{0})} = 0$. Define $\circ : \mathcal{A} \times \mathcal{A} \rightarrow \mathcal{A}$ by
\[x \circ y = \rm{[\![\it{x},\it{y},\it{z_{0}}]\!]},\quad\forall \it{x}, \it{y} \in \mathcal{A},\]
then $J = (\mathcal{A}, \circ)$ is a Jordan algebra with $z_{0} \in C(J)$, the center of $J$.
\end{prop}
\begin{proof}
For any $x, y, u , v\in \mathcal{A}$, we have
\begin{align*}
&[R_{x}, R_{y}](u \circ v) = \rm{[\![\rm{[\![\rm{[\![\it{u},\it{v},\it{z_{0}}]\!]},\it{x},\it{z_{0}}]\!]},\it{y},\it{z_{0}}]\!]} - \rm{[\![\rm{[\![\rm{[\![\it{u},\it{v},\it{z_{0}}]\!]},\it{y},\it{z_{0}}]\!]},\it{x},\it{z_{0}}]\!]}\\
&= \rm{[\![\rm{[\![\rm{[\![\it{u},\it{x},\it{z_{0}}]\!]},\it{y},\it{z_{0}}]\!]},\it{v},\it{z_{0}}]\!]} - \rm{[\![\rm{[\![\rm{[\![\it{u},\it{y},\it{z_{0}}]\!]},\it{x},\it{z_{0}}]\!]},\it{v},\it{z_{0}}]\!]} + \rm{[\![\it{u},\rm{[\![\rm{[\![\it{v},\it{x},\it{z_{0}}]\!]},\it{y},\it{z_{0}}]\!]},\it{z_{0}}]\!]}\\
&- \rm{[\![\it{u},\rm{[\![\rm{[\![\it{v},\it{y},\it{z_{0}}]\!]},\it{x},\it{z_{0}}]\!]},\it{z_{0}}]\!]} + \rm{[\![\it{u},\it{v},\rm{[\![\rm{[\![\it{z_{0}},\it{x},\it{z_{0}}]\!]},\it{y},\it{z_{0}}]\!]}]\!]} - \rm{[\![\it{u},\it{v},\rm{[\![\rm{[\![\it{z_{0}},\it{y},\it{z_{0}}]\!]},\it{x},\it{z_{0}}]\!]}]\!]}\\
&= \rm{[\![\rm{[\![\rm{[\![\it{u},\it{x},\it{z_{0}}]\!]},\it{y},\it{z_{0}}]\!]},\it{v},\it{z_{0}}]\!]} - \rm{[\![\rm{[\![\rm{[\![\it{u},\it{y},\it{z_{0}}]\!]},\it{x},\it{z_{0}}]\!]},\it{v},\it{z_{0}}]\!]} + \rm{[\![\it{u},\rm{[\![\rm{[\![\it{v},\it{x},\it{z_{0}}]\!]},\it{y},\it{z_{0}}]\!]},\it{z_{0}}]\!]}\\
&- \rm{[\![\it{u},\rm{[\![\rm{[\![\it{v},\it{y},\it{z_{0}}]\!]},\it{x},\it{z_{0}}]\!]},\it{z_{0}}]\!]}\\
&= [R_{x}, R_{y}](u) \circ v + u \circ [R_{x}, R_{y}](v),
\end{align*}
which implies that $[R_{x}, R_{y}] \in Der(J)$, i.e., $J = (\mathcal{A}, \circ)$ is a Jordan algebra. It's obvious that $z_{0} \in C(J)$ since $z_{0} \circ x  = 0$ for any $x \in \mathcal{A}$.
\end{proof}
\begin{lem}\label{le:3.3}
Let $\mathscr{A}$ be a commutative associative algebra over $\rm{F}$ with $char \rm{F} \neq 2$ and $\Delta \in Der(\mathscr{A})$. Suppose that $\omega$ is an involution on $\mathscr{A}$. Define two multiplications $\circ_{\Delta}$, $\circ_{\omega}$ from $\mathscr{A} \times \mathscr{A}$ to $\mathscr{A}$ by
\[a \circ_{\Delta} b = a\Delta(b) + b\Delta(a),\quad a \circ_{\omega} b = a\omega(b) + b\omega(a),\quad\forall a, b \in \mathscr{A}.\]
Then
\begin{enumerate}[(1)]
\item $(\mathscr{A}, \circ_{\Delta})$ is a Jordan algebra if and only if for any $a, b \in \mathscr{A}$, $\Delta$ satisfies
      \begin{equation}\label{eq:3.2}
       \Delta(ba^{2}\Delta^{2}(a)) = 0.\tag{3.2}
      \end{equation}
\item $(\mathscr{A}, \circ_{\omega})$ is a Jordan algebra if and only if for any $a, b \in \mathscr{A}$, $\omega$ satisfies
      \begin{equation}\label{eq:3.3}
       (a - \omega(a))(b - \omega(b)) = 0.\tag{3.3}
      \end{equation}
\end{enumerate}
\end{lem}
\begin{proof}
It follows from a direct computation.
\end{proof}
\begin{ex}
Let $\mathscr{A}$ denote the algebra $\rm{F}[\it{t}]/\langle{t}^{3}\rangle$ over $\rm{F}$ where $char \rm{F} = 3$. For convenience, let $t^{i}$ denote the congruence class of $t^{i}$ in $\rm{F}[\it{t}]/\langle\it{t}^{3}\rangle$. Then $\mathscr{A}$ has a basis $\{\rm{1}, \it{t}, \it{t}^{2}\}$. Define $D : \mathscr{A} \rightarrow \mathscr{A}$ to be a linear map by
$$D(z) =
\left\{
\begin{aligned}
a_{0}\rm{1} + \it{a_{1}t} + \it{a_{2}t^{2}},\quad \it{z} = \rm{1},\\
b_{0}\rm{1} + \it{b_{1}t} + \it{b_{2}t^{2}},\quad \it{z} = \it{t},\\
c_{0}\rm{1} + \it{c_{1}t} + \it{c_{2}t^{2}},\quad \it{z} = \it{t^{2}}.
\end{aligned}
\right.
$$
Then while
$$
\left\{
\begin{aligned}
a_{0} = a_{1} = a_{2} = 0,\\
b_{0} + c_{0} = b_{1} + c_{1} = b_{2} + c_{2} = 0,\\
b_{0} = 2a_{2}, b_{1} = 2a_{0}, b_{2} = 2a_{1},\\
a_{0} = 2b_{1}, a_{1} = 2b_{2}, a_{2} = 2b_{0},
\end{aligned}
\right.
$$
$D \in Der(\mathscr{A})$. Note that $char \rm{F} = 3$, we get $b_{2} = -2b_{0} = -2b_{1}$, $a_{0} = a_{2} = -2a_{1}$, $a_{1} = -b_{0}$. Moreover, we get $D^{2}(x) = 0$ for any $x \in \mathscr{A}$, which implies that $D$ satisfies (\ref{eq:3.2}).
\end{ex}
\begin{ex}
Let $\mathscr{A}$ denote the commutative associative algebra with a basis $\{e, u\}$ over $\rm{F}$ where $char \rm{F} \neq 2$. And the multiplication table is
\[ee = e, eu = ue = e, uu = 0.\]
Define $\omega$ to be a linear map by
$$\omega(z) =
\left\{
\begin{aligned}
e,\quad z = e,\\
0,\quad z = u.
\end{aligned}
\right.
$$
One can verify that $\omega$ is an involution on $\mathscr{A}$ satisfying (\ref{eq:3.3}).
\end{ex}
\begin{prop}
Let $\mathscr{A}$ be a commutative associative algebra over $\rm{F}$ with $char \rm{F} \neq 2$ and $\Delta$ a derivation of $\mathscr{A}$ satisfying (\ref{eq:3.2}). Suppose that $\omega$ is an involution on $\mathscr{A}$ satisfying (\ref{eq:3.3}). If $\alpha, \beta \in \mathscr{A}^{*}$ satisfy
\[\alpha(ab) = 0,\;\beta\Delta(ab) = 0,\quad\forall a, b \in \mathscr{A},\]
then $(\mathscr{A}, \rm{[\![\cdot,\cdot,\cdot]\!]}_{\alpha, \omega})$ and $(\mathscr{A}, \rm{[\![\cdot,\cdot,\cdot]\!]}_{\beta, \Delta})$ are two ternary Jordan algebras where
\[\rm{[\![\it{a},\it{b},\it{c}]\!]}_{\alpha, \omega} = \it{\alpha(a)b \circ_{\omega} c + \alpha(b)c \circ_{\omega} a + \alpha(c)a \circ_{\omega} b},\quad\forall \it{a, b, c} \in \mathscr{A}.\]
\[\rm{[\![\it{a},\it{b},\it{c}]\!]}_{\beta, \Delta} = \it{\beta(a)b \circ_{\Delta} c + \beta(b)c \circ_{\Delta} a + \beta(c)a \circ_{\Delta} b},\quad\forall \it{a, b, c} \in \mathscr{A}.\]
\end{prop}
\begin{proof}
For every $c \in \mathscr{A}$,
\[\omega(c + \omega(c)) = c + \omega(c),\;\omega(c - \omega(c)) = -(c - \omega(c)),\;and\; c = \frac{1}{2}(c + \omega(c)) + \frac{1}{2}(c - \omega(c)).\]
Therefore,
\[\mathscr{A} = \mathscr{A}_{1} \dotplus \mathscr{A}_{2}\]
where
\[\mathscr{A}_{1} = \{a \in \mathscr{A} \mid \omega(a) = a\},\;\mathscr{A}_{2} = \{a \in \mathscr{A} \mid \omega(a) = -a\}.\]

Hence, we have
$$\alpha(a\omega(b) + b\omega(a)) =
\left\{
\begin{aligned}
0,\quad a \in \mathscr{A}_{1}, b \in \mathscr{A}_{2}\; or\; a \in \mathscr{A}_{2}, b \in \mathscr{A}_{1},\\
2\alpha(ab),\quad a, b \in \mathscr{A}_{1}\; or\; a, b \in \mathscr{A}_{2}.
\end{aligned}
\right.
$$
According to Proposition \ref{prop:3.1} and Lemma \ref{le:3.3}, the results hold.
\end{proof}
\begin{prop}
Suppose that $\mathcal{A}$ is a ternary Jordan algebra over $\rm{F}$ and $\mathscr{A}$ is a commutative associative algebra. On the vector space $\tilde{\mathcal{A}} = \mathcal{A} \otimes \mathscr{A}$ define a new multiplication by
\[\rm{[\![\it{x \otimes a},\it{y \otimes b},\it{z \otimes c}]\!]}^{'} = \rm{[\![\it{x},\it{y},\it{z}]\!]} \otimes \it{abc},\;\forall \it{x, y, z} \in \mathcal{A}, \it{a, b, c} \in \mathscr{A}.\]
Then $(\tilde{\mathcal{A}}, \rm{[\![\cdot,\cdot,\cdot]\!]}^{'})$ is also a ternary Jordan algebra. Moreover, if $\mathcal{A}$ is perfect and $\mathscr{A}$ is unital, then $\tilde{\mathcal{A}}$ is also perfect.
\end{prop}
\begin{proof}
It follows a direct computation.
\end{proof}
\section{Generalized derivations of ternary Jordan algebras}\label{se:4}
In this section, we study the relations between $GDer(\mathcal{A})$, $QDer(\mathcal{A})$, $Der(\mathcal{A})$, $ZDer(\mathcal{A})$, $Q\Gamma(\mathcal{A})$ and $\Gamma(\mathcal{A})$ of a ternary Jordan algebra $\mathcal{A}$. And we suppose that $char \rm{F} \neq 2, 3$ in this section.

\begin{thm}\label{thm:4.1}
Let $\mathcal{A}$ be a ternary Jordan algebra. Then $GDer(\mathcal{A})$, $QDer(\mathcal{A})$, $Q\Gamma(\mathcal{A})$ and $\Gamma(\mathcal{A})$ are subalgebras of the general linear Lie algebra $gl(\mathcal{A})$. In addition, if $Z(\mathcal{A}) = \{0\}$, then $Q\Gamma(\mathcal{A})$ and $\Gamma(\mathcal{A})$ are abelian.
\end{thm}
\begin{proof}
For arbitrary $f_{1}, g_{1} \in GDer(\mathcal{A})$ and $x, y, z \in \mathcal{A}$, we have
\begin{align*}
&\rm{[\![[\it{f}_{1}, \it{g}_{1}](\it{x}),\it{y},\it{z}]\!]} = \rm{[\![\it{f}_{1}\it{g}_{1}(\it{x}),\it{y},\it{z}]\!]} - \rm{[\![\it{g}_{1}\it{f}_{1}(\it{x}),\it{y},\it{z}]\!]}\\
&= f^{'}(\rm{[\![\it{g}_{1}(\it{x}),\it{y},\it{z}]\!]}) - \rm{[\![\it{g}_{1}(\it{x}),\it{f}_{2}(\it{y}),\it{z}]\!]} - \rm{[\![\it{g}_{1}(\it{x}),\it{y},\it{f}_{3}(\it{z})]\!]} - g^{'}(\rm{[\![\it{f}_{1}(\it{x}),\it{y},\it{z}]\!]}) + \rm{[\![\it{f}_{1}(\it{x}),\it{g}_{2}(\it{y}),\it{z}]\!]}\\
&+ \rm{[\![\it{g}_{1}(\it{x}),\it{y},\it{g}_{3}(\it{z})]\!]}\\
&= f^{'}(g^{'}(\rm{[\![\it{x},\it{y},\it{z}]\!]}) - \rm{[\![\it{x},\it{g}_{2}(\it{y}),\it{z}]\!]} - \rm{[\![\it{x},\it{y},\it{g}_{3}(\it{z})]\!]}) - \rm{[\![\it{g}_{1}(\it{x}),\it{f}_{2}(\it{y}),\it{z}]\!]} - \rm{[\![\it{g}_{1}(\it{x}),\it{y},\it{f}_{3}(\it{z})]\!]}\\
&- g^{'}(f^{'}(\rm{[\![\it{x},\it{y},\it{z}]\!]}) - \rm{[\![\it{x},\it{f}_{2}(\it{y}),\it{z}]\!]} - \rm{[\![\it{x},\it{y},\it{f}_{3}(\it{z})]\!]})+ \rm{[\![\it{f}_{1}(\it{x}),\it{g}_{2}(\it{y}),\it{z}]\!]} + \rm{[\![\it{g}_{1}(\it{x}),\it{y},\it{g}_{3}(\it{z})]\!]}\\
&= f^{'}g^{'}(\rm{[\![\it{x},\it{y},\it{z}]\!]}) - \rm{[\![\it{f}_{1}(\it{x}),\it{g}_{2}(\it{y}),\it{z}]\!]} - \rm{[\![\it{x},\it{f}_{2}\it{g}_{2}(\it{y}),\it{z}]\!]} - \rm{[\![\it{x},\it{g}_{2}(\it{y}),\it{f}_{3}(\it{z})]\!]} - \rm{[\![\it{f}_{1}(\it{x}),\it{y},\it{g}_{3}(\it{z})]\!]}\\
&- \rm{[\![\it{x},\it{f}_{2}(\it{y}),\it{g}_{3}(\it{z})]\!]} - \rm{[\![\it{x},\it{y},\it{f}_{3}\it{g}_{3}(\it{z})]\!]} - \rm{[\![\it{g}_{1}(\it{x}),\it{f}_{2}(\it{y}),\it{z}]\!]} - \rm{[\![\it{g}_{1}(\it{x}),\it{y},\it{f}_{3}(\it{z})]\!]}\\
&- g^{'}f^{'}(\rm{[\![\it{x},\it{y},\it{z}]\!]}) + \rm{[\![\it{g}_{1}(\it{x}),\it{f}_{2}(\it{y}),\it{z}]\!]} + \rm{[\![\it{x},\it{g}_{2}\it{f}_{2}(\it{y}),\it{z}]\!]} + \rm{[\![\it{x},\it{f}_{2}(\it{y}),\it{g}_{3}(\it{z})]\!]} + \rm{[\![\it{g}_{1}(\it{x}),\it{y},\it{f}_{3}(\it{z})]\!]}\\
&+ \rm{[\![\it{x},\it{g}_{2}(\it{y}),\it{f}_{3}(\it{z})]\!]} + \rm{[\![\it{x},\it{y},\it{g}_{3}\it{f}_{3}(\it{z})]\!]} + \rm{[\![\it{f}_{1}(\it{x}),\it{g}_{2}(\it{y}),\it{z}]\!]} + \rm{[\![\it{f}_{1}(\it{x}),\it{y},\it{g}_{3}(\it{z})]\!]}\\
&= [f^{'}, g^{'}](\rm{[\![\it{x},\it{y},\it{z}]\!]}) - \rm{[\![\it{x},[\it{f}_{2}, \it{g}_{2}](\it{y}),\it{z}]\!]} - \rm{[\![\it{x},\it{y},[\it{f}_{3}, \it{g}_{3}](\it{z})]\!]}.
\end{align*}
Obviously, $[f^{'}, g^{'}], [f_{2}, g_{2}], [f_{3}, g_{3}] \in \rm{End}(\mathcal{A})$. Hence, $[f_{1}, g_{1}] \in GDer(\mathcal{A})$, i.e., $GDer(\mathcal{A})$ is a subalgebra of $gl(\mathcal{A})$.

Similarly, $QDer(\mathcal{A})$ is a subalgebra of $gl(\mathcal{A})$.

For arbitrary $f, g \in Q\Gamma(\mathcal{A})$ and $x, y, z \in \mathcal{A}$, we have
\begin{align*}
&\rm{[\![[\it{f}, \it{g}](\it{x}),\it{y},\it{z}]\!]} = \rm{[\![\it{f}\it{g}(\it{x}),\it{y},\it{z}]\!]} - \rm{[\![\it{g}\it{f}(\it{x}),\it{y},\it{z}]\!]} = \rm{[\![\it{g}(\it{x}),\it{y},\it{f}(\it{z})]\!]} - \rm{[\![\it{f}(\it{x}),\it{y},\it{g}(\it{z})]\!]}\\
&= \rm{[\![\it{x},\it{g}(\it{y}),\it{f}(\it{z})]\!]} - \rm{[\![\it{x},\it{f}(\it{y}),\it{g}(\it{z})]\!]} = \rm{[\![\it{x},\it{f}\it{g}(\it{y}),\it{z}]\!]} - \rm{[\![\it{x},\it{g}\it{f}(\it{y}),\it{z}]\!]} = \rm{[\![\it{x},[\it{f}, \it{g}](\it{y}),\it{z}]\!]},
\end{align*}
similarly, we have $\rm{[\![[\it{f}, \it{g}](\it{x}),\it{y},\it{z}]\!]} = \rm{[\![\it{x},\it{y},[\it{f}, \it{g}](\it{z})]\!]}$, which implies that $[f, g] \in Q\Gamma(\mathcal{A})$, i.e., $Q\Gamma(\mathcal{A})$ is a subalgebra of $gl(\mathcal{A})$.

For arbitrary $f, g \in \Gamma(\mathcal{A})$ and $x, y, z \in \mathcal{A}$, we have
\begin{align*}
&[f, g](\rm{[\![\it{x},\it{y},\it{z}]\!]}) = \it{fg}(\rm{[\![\it{x},\it{y},\it{z}]\!]}) - \it{gf}(\rm{[\![\it{x},\it{y},\it{z}]\!]}) = \it{g}(\rm{[\![\it{x},\it{f}(\it{y}),\it{z}]\!]}) - \it{f}(\rm{[\![\it{x},\it{g}(\it{y}),\it{z}]\!]})\\
&= \rm{[\![\it{g}(\it{x}),\it{f}(\it{y}),\it{z}]\!]} - \rm{[\![\it{f}(\it{x}),\it{g}(\it{y}),\it{z}]\!]} = \rm{[\![\it{f}\it{g}(\it{x}),\it{y},\it{z}]\!]} - \rm{[\![\it{g}\it{f}(\it{x}),\it{y},\it{z}]\!]} = \rm{[\![[\it{f}, \it{g}](\it{x}),\it{y},\it{z}]\!]},
\end{align*}
similarly, we have $[f, g](\rm{[\![\it{x},\it{y},\it{z}]\!]}) = \rm{[\![\it{x},[\it{f}, \it{g}](\it{y}),\it{z}]\!]} = \rm{[\![\it{x},\it{y},[\it{f}, \it{g}](\it{z})]\!]}$, i.e., $[f, g] \in \Gamma(\mathcal{A})$. Therefore, $\Gamma(\mathcal{A})$ is a subalgebra of $gl(\mathcal{A})$.

For all $f, g \in Q\Gamma(\mathcal{A})$ and $x, y, z \in \mathcal{A}$, we have
\begin{align*}
&\rm{[\![[\it{f}, \it{g}](\it{x}),\it{y},\it{z}]\!]} = \rm{[\![\it{f}\it{g}(\it{x}),\it{y},\it{z}]\!]} - \rm{[\![\it{g}\it{f}(\it{x}),\it{y},\it{z}]\!]} = \rm{[\![\it{g}(\it{x}),\it{f}(\it{y}),\it{z}]\!]} - \rm{[\![\it{f}(\it{x}),\it{g}(\it{y}),\it{z}]\!]}\\
&= \rm{[\![\it{x},\it{f}(\it{y}),\it{g}(\it{z})]\!]} - \rm{[\![\it{f}(\it{x}),\it{y},\it{g}(\it{z})]\!]} = \rm{[\![\it{x},\it{y},\it{f}\it{g}(\it{z})]\!]} - \rm{[\![\it{x},\it{y},\it{f}\it{g}(\it{z})]\!]} = 0,
\end{align*}
which implies that $[\it{f}, \it{g}](\it{x}) \in Z(\mathcal{A})$. Since $Z(\mathcal{A}) = \{0\}$, we have $[\it{f}, \it{g}](\it{x}) = \rm{0}$ for all $x \in \mathcal{A}$, so $[\it{f}, \it{g}] = \rm{0}$, i.e., $Q\Gamma(\mathcal{A})$ is abelian. Similarly, $\Gamma(\mathcal{A})$ is abelian.
\end{proof}
\begin{re}\label{re:4.2}
\begin{enumerate}[(1)]
\item According to Theorem \ref{thm:4.1}, we see that $Q\Gamma(\mathcal{A})$ is a subalgebra of $gl(\mathcal{A})$ where $\mathcal{A}$ denotes a ternary Jordan algebra, which is different from in the case of Jordan algebras. As an application, we also get $Q\Gamma(\mathcal{A})$ is an ideal of $GDer(\mathcal{A})$(See Theorem \ref{thm:4.4}).
\item According to the proof of Theorem \ref{thm:4.1}, we have $[Q\Gamma(\mathcal{A}), Q\Gamma(\mathcal{A})]$, $[Q\Gamma(\mathcal{A}), \Gamma(\mathcal{A})]$ and $[\Gamma(\mathcal{A}), \Gamma(\mathcal{A})] \in \rm{End}(\mathcal{A}, Z(\mathcal{A}))$.
\end{enumerate}
\end{re}
\begin{prop}\label{prop:4.3}
Let $\mathcal{A}$ be a ternary Jordan algebra, then
\begin{enumerate}[(1)]
\item $[Der(\mathcal{A}), \Gamma(\mathcal{A})] \subseteq \Gamma(\mathcal{A})$;
\item $[QDer(\mathcal{A}), Q\Gamma(\mathcal{A})] \subseteq Q\Gamma(\mathcal{A})$;
\item $[Q\Gamma(\mathcal{A}), Q\Gamma(\mathcal{A})] \subseteq QDer(\mathcal{A})$;
\item $\Gamma(\mathcal{A}) \subseteq QDer(\mathcal{A})$;
\item $Q\Gamma(\mathcal{A}) \subseteq GDer(\mathcal{A})$;
\item $\Gamma(\mathcal{A})Der(\mathcal{A}) \subseteq Der(\mathcal{A})$;
\item $\Gamma(\mathcal{A}) \subseteq QDer(\mathcal{A}) \cap Q\Gamma(\mathcal{A})$.
\end{enumerate}
\end{prop}
\begin{proof}
(1). For any $f \in Der(\mathcal{A})$, $g \in \Gamma(\mathcal{A})$ and $x, y, z \in \mathcal{A}$, we have
\begin{align*}
&[f, g](\rm{[\![\it{x},\it{y},\it{z}]\!]}) = \it{fg}(\rm{[\![\it{x},\it{y},\it{z}]\!]}) - \it{gf}(\rm{[\![\it{x},\it{y},\it{z}]\!]})\\
&= \it{f}(\rm{[\![\it{g}(\it{x}),\it{y},\it{z}]\!]}) - \it{g}(\rm{[\![\it{f}(\it{x}),\it{y},\it{z}]\!]} + \rm{[\![\it{x},\it{f}(\it{y}),\it{z}]\!]} + \rm{[\![\it{x},\it{y},\it{f}(\it{z})]\!]})\\
&= \rm{[\![\it{f}\it{g}(\it{x}),\it{y},\it{z}]\!]} + \rm{[\![\it{g}(\it{x}),\it{f}(\it{y}),\it{z}]\!]} + \rm{[\![\it{g}(\it{x}),\it{y},\it{f}(\it{z})]\!]} - \rm{[\![\it{g}\it{f}(\it{x}),\it{y},\it{z}]\!]} - \rm{[\![\it{g}(\it{x}),\it{f}(\it{y}),\it{z}]\!]}\\
&- \rm{[\![\it{g}(\it{x}),\it{y},\it{f}(\it{z})]\!]}\\
&= \rm{[\![\it{f}\it{g}(\it{x}),\it{y},\it{z}]\!]} - \rm{[\![\it{g}\it{f}(\it{x}),\it{y},\it{z}]\!]} = \rm{[\![[\it{f}, \it{g}](\it{x}),\it{y},\it{z}]\!]},
\end{align*}
similarly, we have $[f, g](\rm{[\![\it{x},\it{y},\it{z}]\!]}) = \rm{[\![\it{x},[\it{f}, \it{g}](\it{y}),\it{z}]\!]} = \rm{[\![\it{x},\it{y},[\it{f}, \it{g}](\it{z})]\!]}$, which implies that $[f, g] \in \Gamma(\mathcal{A})$. Hence, $[Der(\mathcal{A}), \Gamma(\mathcal{A})] \subseteq \Gamma(\mathcal{A})$.

(2). For any $f \in QDer(\mathcal{A})$, $g \in Q\Gamma(\mathcal{A})$ and $x, y, z \in \mathcal{A}$, we have
\begin{align*}
&\rm{[\![[\it{f}, \it{g}](\it{x}),\it{y},\it{z}]\!]} = \rm{[\![\it{f}\it{g}(\it{x}),\it{y},\it{z}]\!]} - \rm{[\![\it{g}\it{f}(\it{x}),\it{y},\it{z}]\!]}\\
&= \it{f}^{'}(\rm{[\![\it{g}(\it{x}),\it{y},\it{z}]\!]}) - \rm{[\![\it{g}(\it{x}),\it{f}(\it{y}),\it{z}]\!]} - \rm{[\![\it{g}(\it{x}),\it{y},\it{f}(\it{z})]\!]} - \rm{[\![\it{f}(\it{x}),\it{g}(\it{y}),\it{z}]\!]}\\
&= \it{f}^{'}(\rm{[\![\it{x},\it{g}(\it{y}),\it{z}]\!]}) - \rm{[\![\it{x},\it{g}\it{f}(\it{y}),\it{z}]\!]} - \rm{[\![\it{x},\it{g}(\it{y}),\it{f}(\it{z})]\!]} - \it{f}^{'}(\rm{[\![\it{x},\it{g}(\it{y}),\it{z}]\!]}) + \rm{[\![\it{x},\it{f}\it{g}(\it{y}),\it{z}]\!]}\\
&+ \rm{[\![\it{x},\it{g}(\it{y}),\it{f}(\it{z})]\!]}\\
&= \rm{[\![\it{x},[\it{f}, \it{g}](\it{y}),\it{z}]\!]},
\end{align*}
similarly, we have $\rm{[\![[\it{f}, \it{g}](\it{x}),\it{y},\it{z}]\!]} = \rm{[\![\it{x},\it{y},[\it{f}, \it{g}](\it{z})]\!]}$. Hence, $[f, g] \in Q\Gamma(\mathcal{A})$.

So $[QDer(\mathcal{A}), Q\Gamma(\mathcal{A})] \subseteq Q\Gamma(\mathcal{A})$.

(3). For any $f, g \in Q\Gamma(\mathcal{A})$ and $x, y, z \in \mathcal{A}$, according to the proof of Theorem \ref{thm:4.1}, we have
\[\rm{[\![[\it{f}, \it{g}](\it{x}),\it{y},\it{z}]\!]} = 0,\]
Therefore, we have
\[\rm{[\![[\it{f}, \it{g}](\it{x}),\it{y},\it{z}]\!]} + \rm{[\![\it{x},[\it{f}, \it{g}](\it{y}),\it{z}]\!]} + \rm{[\![\it{x},\it{y},[\it{f}, \it{g}](\it{z})]\!]} = 0,\]
which implies that $[f, g] \in QDer(\mathcal{A})$. Hence, $[Q\Gamma(\mathcal{A}), Q\Gamma(\mathcal{A})] \subseteq QDer(\mathcal{A})$.

(4). For any $f \in \Gamma(\mathcal{A})$ and $x, y, z \in \mathcal{A}$, we have
\[\rm{[\![\it{f}(\it{x}),\it{y},\it{z}]\!]} + \rm{[\![\it{x},\it{f}(\it{y}),\it{z}]\!]} + \rm{[\![\it{x},\it{y},\it{f}(\it{z})]\!]} = 3\it{f}(\rm{[\![\it{x},\it{y},\it{z}]\!]}),\]
which implies that $f \in QDer(\mathcal{A})$, so $\Gamma(\mathcal{A}) \subseteq QDer(\mathcal{A})$.

(5). For all $f \in Q\Gamma(\mathcal{A})$ and $x, y, z \in \mathcal{A}$, we have
\[\rm{[\![\it{f}(\it{x}),\it{y},\it{z}]\!]} + \rm{[\![\it{x},-\frac{1}{2}\it{f}(\it{y}),\it{z}]\!]} + \rm{[\![\it{x},\it{y},-\frac{1}{2}\it{f}(\it{z})]\!]} = 0,\]
which implies that $f \in GDer(\mathcal{A})$. Hence, $Q\Gamma(\mathcal{A}) \subseteq GDer(\mathcal{A})$.

(6). For any $f \in \Gamma(\mathcal{A})$, $g \in Der(\mathcal{A})$ and $x, y, z \in \mathcal{A}$, we have
\begin{align*}
&\it{fg}(\rm{[\![\it{x},\it{y},\it{z}]\!]}) = \it{f}(\rm{[\![\it{g}(\it{x}),\it{y},\it{z}]\!]} + \rm{[\![\it{x},\it{g}(\it{y}),\it{z}]\!]} + \rm{[\![\it{x},\it{y},\it{g}(\it{z})]\!]})\\
&= \rm{[\![\it{f}\it{g}(\it{x}),\it{y},\it{z}]\!]} + \rm{[\![\it{x},\it{f}\it{g}(\it{y}),\it{z}]\!]} + \rm{[\![\it{x},\it{y},\it{f}\it{g}(\it{z})]\!]},
\end{align*}
which implies that $fg \in Der(\mathcal{A})$. So $\Gamma(\mathcal{A})Der(\mathcal{A}) \subseteq Der(\mathcal{A})$.

(7). Obviously, $\Gamma(\mathcal{A}) \subseteq Q\Gamma(\mathcal{A})$. According to (4), we have $\Gamma(\mathcal{A}) \subseteq QDer(\mathcal{A}) \cap Q\Gamma(\mathcal{A})$.
\end{proof}
\begin{thm}\label{thm:4.4}
Let $\mathcal{A}$ be a ternary Jordan algebra over $\rm{F}$. Then
\begin{enumerate}[(1)]
\item $GDer(\mathcal{A}) = QDer(\mathcal{A}) + Q\Gamma(\mathcal{A})$;
\item $Q\Gamma(\mathcal{A})$ is an ideal of $GDer(\mathcal{A})$. In addition, if $Z(\mathcal{A}) = \{0\}$, then $Q\Gamma(\mathcal{A})$ is an abelian ideal of $GDer(\mathcal{A})$.
\end{enumerate}
\end{thm}
\begin{proof}
(1). According to Proposition \ref{prop:4.3}, we only need to show that $GDer(\mathcal{A}) \subseteq QDer(\mathcal{A}) + Q\Gamma(\mathcal{A})$.

For any $f_{1} \in GDer(\mathcal{A})$, and $(f_{1}, f_{2}, f_{3}, f^{'}) \in \Delta(\mathcal{A})$, we see that

$(f_{1}, f_{2}, f_{3}, f^{'}) = (\frac{f_{1} + f_{2} + f_{3}}{3}, \frac{f_{1} + f_{2} + f_{3}}{3}, \frac{f_{1} + f_{2} + f_{3}}{3}, f^{'}) + (\frac{2f_{1} - f_{2} - f_{3}}{3}, \frac{2f_{2} - f_{1} - f_{3}}{3}, \frac{2f_{3} - f_{1} - f_{2}}{3}, 0)$.

Now we prove that $\frac{f_{1} + f_{2} + f_{3}}{3} \in QDer(\mathcal{A})$, and maps $\frac{2f_{1} - f_{2} - f_{3}}{3}, \frac{2f_{2} - f_{1} - f_{3}}{3}, \frac{2f_{3} - f_{1} - f_{2}}{3} \in Q\Gamma(\mathcal{A})$.

For any $x, y, z \in \mathcal{A}$, we have
\begin{align*}
&\rm{[\![\frac{\it{f}_{1} + \it{f}_{2} + \it{f}_{3}}{3}(\it{x}),\it{y},\it{z}]\!]} + \rm{[\![\it{x},\frac{f_{1} + f_{2} + f_{3}}{3}(\it{y}),\it{z}]\!]} + \rm{[\![\it{x},\it{y},\frac{f_{1} + f_{2} + f_{3}}{3}(\it{z})]\!]}\\
&= \frac{1}{3}(\rm{[\![\it{f}_{1}(\it{x}),\it{y},\it{z}]\!]} + \rm{[\![\it{x},\it{f}_{2}(\it{y}),\it{z}]\!]} + \rm{[\![\it{x},\it{y},\it{f}_{3}(\it{z})]\!]}) + \frac{1}{3}(\rm{[\![\it{f}_{2}(\it{x}),\it{y},\it{z}]\!]} + \rm{[\![\it{x},\it{f}_{3}(\it{y}),\it{z}]\!]} + \rm{[\![\it{x},\it{y},\it{f}_{1}(\it{z})]\!]})\\
&+ \frac{1}{3}(\rm{[\![\it{f}_{3}(\it{x}),\it{y},\it{z}]\!]} + \rm{[\![\it{x},\it{f}_{1}(\it{y}),\it{z}]\!]} + \rm{[\![\it{x},\it{y},\it{f}_{2}(\it{z})]\!]})\\
&= f^{'}(\rm{[\![\it{x},\it{y},\it{z}]\!]}).
\end{align*}
It follows that $\frac{f_{1} + f_{2} + f_{3}}{3} \in QDer(\mathcal{A})$.
\begin{align*}
&\rm{[\![\frac{2\it{f}_{1} - \it{f}_{2} - \it{f}_{3}}{3}(\it{x}),\it{y},\it{z}]\!]} - \rm{[\![\it{x},\frac{2f_{1} - f_{2} - f_{3}}{3}(\it{y}),\it{z}]\!]}\\
&= \frac{2}{3}\rm{[\![\it{f}_{1}(\it{x}),\it{y},\it{z}]\!]} - \frac{1}{3}\rm{[\![\it{f}_{2}(\it{x}),\it{y},\it{z}]\!]} - \frac{1}{3}\rm{[\![\it{f}_{3}(\it{x}),\it{y},\it{z}]\!]} - \frac{2}{3}\rm{[\![\it{x},\it{f}_{1}(\it{y}),\it{z}]\!]} + \frac{1}{3}\rm{[\![\it{x},\it{f}_{2}(\it{y}),\it{z}]\!]}\\
&+ \frac{1}{3}\rm{[\![\it{x},\it{f}_{3}(\it{y}),\it{z}]\!]}\\
&= \frac{2}{3}\it{f}^{'}(\rm{[\![\it{x},\it{y},\it{z}]\!]}) - \frac{2}{3}\rm{[\![\it{x},\it{f}_{2}(\it{y}),\it{z}]\!]} - \frac{2}{3}\rm{[\![\it{x},\it{y},\it{f}_{3}(\it{z})]\!]} - \frac{1}{3}\it{f}^{'}(\rm{[\![\it{x},\it{y},\it{z}]\!]}) + \frac{1}{3}\rm{[\![\it{x},\it{f}_{1}(\it{y}),\it{z}]\!]}\\
&+ \frac{1}{3}\rm{[\![\it{x},\it{y},\it{f}_{3}(\it{z})]\!]} - \frac{1}{3}\it{f}^{'}(\rm{[\![\it{x},\it{y},\it{z}]\!]}) + \frac{1}{3}\rm{[\![\it{x},\it{f}_{2}(\it{y}),\it{z}]\!]} + \frac{1}{3}\rm{[\![\it{x},\it{y},\it{f}_{1}(\it{z})]\!]} - \frac{2}{3}\rm{[\![\it{x},\it{f}_{1}(\it{y}),\it{z}]\!]}\\
&+ \frac{1}{3}\rm{[\![\it{x},\it{f}_{2}(\it{y}),\it{z}]\!]} + \frac{1}{3}\rm{[\![\it{x},\it{f}_{3}(\it{y}),\it{z}]\!]}\\
&= -\frac{1}{3}\rm{[\![\it{x},\it{f}_{1}(\it{y}),\it{z}]\!]} - \frac{1}{3}\rm{[\![\it{x},\it{y},\it{f}_{3}(\it{z})]\!]} + \frac{1}{3}\rm{[\![\it{x},\it{y},\it{f}_{1}(\it{z})]\!]} + \frac{1}{3}\rm{[\![\it{x},\it{f}_{3}(\it{y}),\it{z}]\!]}\\
&= \frac{1}{3}\rm{[\![\it{x},\it{y},(\it{f}_{1} - \it{f}_{3})(\it{z})]\!]} + \frac{1}{3}\rm{[\![\it{x},(\it{f}_{3} - \it{f}_{1})(\it{y}),\it{z}]\!]},
\end{align*}
note that
\[\it{f}^{'}(\rm{[\![\it{x},\it{y},\it{z}]\!]}) = \rm{[\![\it{f}_{2}(\it{x}),\it{y},\it{z}]\!]} + \rm{[\![\it{x},\it{f}_{1}(\it{y}),\it{z}]\!]} + \rm{[\![\it{x},\it{y},\it{f}_{3}(\it{z})]\!]},\]
\[\it{f}^{'}(\rm{[\![\it{x},\it{y},\it{z}]\!]}) = \rm{[\![\it{f}_{2}(\it{x}),\it{y},\it{z}]\!]} + \rm{[\![\it{x},\it{f}_{3}(\it{y}),\it{z}]\!]} + \rm{[\![\it{x},\it{y},\it{f}_{1}(\it{z})]\!]},\]
we have $\frac{1}{3}\rm{[\![\it{x},\it{y},(\it{f}_{1} - \it{f}_{3})(\it{z})]\!]} + \frac{1}{3}\rm{[\![\it{x},(\it{f}_{3} - \it{f}_{1})(\it{y}),\it{z}]\!]} = 0$.

So $\rm{[\![\frac{2\it{f}_{1} - \it{f}_{2} - \it{f}_{3}}{3}(\it{x}),\it{y},\it{z}]\!]} = \rm{[\![\it{x},\frac{2f_{1} - f_{2} - f_{3}}{3}(\it{y}),\it{z}]\!]}$. Similarly, we have $\rm{[\![\frac{2\it{f}_{1} - \it{f}_{2} - \it{f}_{3}}{3}(\it{x}),\it{y},\it{z}]\!]} = \rm{[\![\it{x},\it{y},\frac{2f_{1} - f_{2} - f_{3}}{3}(\it{z})]\!]}$. Therefore, $\frac{2f_{1} - f_{2} - f_{3}}{3} \in Q\Gamma(\mathcal{A})$. Similarly, we have $\frac{2f_{2} - f_{1} - f_{3}}{3}, \frac{2f_{3} - f_{1} - f_{2}}{3} \in Q\Gamma(\mathcal{A})$.

The results (1) follows.

(2). $[Q\Gamma(\mathcal{A}), GDer(\mathcal{A})] = [Q\Gamma(\mathcal{A}), QDer(\mathcal{A}) + Q\Gamma(\mathcal{A})] = [Q\Gamma(\mathcal{A}), QDer(\mathcal{A})] + [Q\Gamma(\mathcal{A}), Q\Gamma(\mathcal{A})] \subseteq Q\Gamma(\mathcal{A})$, hence $Q\Gamma(\mathcal{A})$ is an ideal of $GDer(\mathcal{A})$. According to Theorem \ref{thm:4.1}, if $Z(\mathcal{A}) = \{0\}$, $Q\Gamma(\mathcal{A})$ is abelian.
\end{proof}
\begin{thm}\label{thm:4.5}
Suppose that a ternary Jordan algebra $\mathcal{A}$ has a decomposition $\mathcal{A} = \mathcal{B} \oplus \mathcal{C}$ where $\mathcal{B}, \mathcal{C}$ are ideals of $\mathcal{A}$. Then we have
\begin{enumerate}[(1)]
\item $Z(\mathcal{A}) = Z(\mathcal{B}) \oplus Z(\mathcal{C})$;
\item If $Z(\mathcal{A}) = \{0\}$, then
     \begin{enumerate}[(a)]
     \item $Der(\mathcal{A}) = Der(\mathcal{B}) \oplus Der(\mathcal{C})$;
     \item $GDer(\mathcal{A}) = GDer(\mathcal{B}) \oplus GDer(\mathcal{C})$;
     \item $QDer(\mathcal{A}) = QDer(\mathcal{B}) \oplus QDer(\mathcal{C})$;
     \item $Q\Gamma(\mathcal{A}) = Q\Gamma(\mathcal{B}) \oplus Q\Gamma(\mathcal{C})$;
     \item $\Gamma(\mathcal{A}) = \Gamma(\mathcal{B}) \oplus \Gamma(\mathcal{C})$.
     \end{enumerate}
\end{enumerate}
\end{thm}
\begin{proof}
(1). Obviously, $Z(\mathcal{B}) \cap Z(\mathcal{C}) = \{0\}$. It follows from a direct computation that $Z(\mathcal{B})$ and $Z(\mathcal{C})$ are ideals of $Z(\mathcal{A})$. Therefore, we have $Z(\mathcal{B}) \oplus Z(\mathcal{C}) \subseteq Z(\mathcal{A})$.

On the other hand, for any $z \in Z(\mathcal{A})$, suppose that $z = z_{1} + z_{2}$ where $z_{1} \in \mathcal{B}$ and $z_{2} \in \mathcal{C}$. Then for any $u_{1}, v_{1} \in \mathcal{B}$, we have
\[\rm{[\![\it{z_{1}}, \it{u_{1}}, \it{v_{1}}]\!]} = \rm{[\![\it{z - z_{2}}, \it{u_{1}}, \it{v_{1}}]\!]} = 0,\]
which implies that $z_{1} \in Z(\mathcal{B})$. Similarly, $z_{2} \in Z(\mathcal{C})$. Hence, $Z(\mathcal{A}) \subseteq Z(\mathcal{B}) \oplus Z(\mathcal{C})$.

Therefore, we have $Z(\mathcal{A}) = Z(\mathcal{B}) \oplus Z(\mathcal{C})$.

(2)(a). For any $D_{1} \in Der(\mathcal{B})$, extend it to a linear map on $\mathcal{A}$ by $D_{1}(x_{1} + x_{2}) = D_{1}(x_{1})$ for any $x_{1} \in \mathcal{B}, x_{2} \in \mathcal{C}$. It's easy to show that $D_{1} \in Der(\mathcal{A})$. So $Der(\mathcal{B}) \subseteq Der(\mathcal{A})$ and $D \in Der(\mathcal{B})$ if and only if $D(\mathcal{C}) = 0$. Similarly, we have $Der(\mathcal{C}) \subseteq Der(\mathcal{A})$ and $D \in Der(\mathcal{C})$ if and only if $D(\mathcal{B}) = 0$. Therefore, we have $Der(\mathcal{B}) + Der(\mathcal{C}) \subseteq Der(\mathcal{A})$ and $Der(\mathcal{B}) \cap Der(\mathcal{C}) = \{0\}$.

On the other hand, suppose that $D \in Der(\mathcal{A})$. For any $a_{1} \in \mathcal{B}$, $b_{2}, c_{2} \in \mathcal{C}$, we have
\begin{align*}
&\rm{[\![\it{D(a_{1})}, \it{b_{2}}, \it{c_{2}}]\!]} = D(\rm{[\![\it{a_{1}}, \it{b_{2}}, \it{c_{2}}]\!]}) - \rm{[\![\it{a_{1}}, \it{D(b_{2})}, \it{c_{2}}]\!]} - \rm{[\![\it{a_{1}}, \it{b_{2}}, \it{D(c_{2})}]\!]}\\
&= - \rm{[\![\it{a_{1}}, \it{D(b_{2})}, \it{c_{2}}]\!]} - \rm{[\![\it{a_{1}}, \it{b_{2}}, \it{D(c_{2})}]\!]} \subseteq \mathcal{B} \cap \mathcal{C} = 0.
\end{align*}
Suppose that $D(a_{1}) = z_{1} + z_{2}$ where $z_{1} \in \mathcal{B}, z_{2} \in \mathcal{C}$. Then
\[0 = \rm{[\![\it{D(a_{1})}, \it{b_{2}}, \it{c_{2}}]\!]} = \rm{[\![\it{z_{1}}, \it{b_{2}}, \it{c_{2}}]\!]} + \rm{[\![\it{z_{2}}, \it{b_{2}}, \it{c_{2}}]\!]} = \rm{[\![\it{z_{2}}, \it{b_{2}}, \it{c_{2}}]\!]},\]
which implies that $z_{2} \in Z(\mathcal{C})$. Note that $Z(\mathcal{A}) = \{0\}$, $Z(\mathcal{C}) = \{0\}$, so $z_{2} = 0$. Hence $D(a_{1}) \in \mathcal{B}$, i.e., $D(\mathcal{B}) \subseteq \mathcal{B}$. Similarly, we have $D(\mathcal{C}) \subseteq \mathcal{C}$.

Next we'll show that $Der(\mathcal{B})$ is an ideal of $Der(\mathcal{A})$. Suppose that $D_{1} \in Der(\mathcal{B})$, $D \in Der(\mathcal{A})$ and $x_{2} \in \mathcal{C}$. Then
\[[D_{1}, D](x_{2}) = D_{1}D(x_{2}) - DD_{1}(x_{2}) = 0,\]
which implies that $[D_{1}, D] \in Der(\mathcal{B})$, i.e., $Der(\mathcal{B})$ is an ideal of $Der(\mathcal{A})$. Similarly, we have $Der(\mathcal{C}) \triangleleft Der(\mathcal{A})$.

In the last, we show that $Der(\mathcal{A}) = Der(\mathcal{B}) \oplus Der(\mathcal{C})$. For any $D \in Der(\mathcal{A})$, suppose that $x = x_{1} + x_{2}$ where $x_{1} \in \mathcal{B}, x_{2} \in \mathcal{C}$. Define $D_{1}, D_{2}$ as follows:
\begin{equation}
\left\{
\begin{aligned}
D_{1}(x_{1} + x_{2}) = D(x_{1}),\\
D_{2}(x_{1} + x_{2}) = D(x_{2}).
\end{aligned}
\right.
\end{equation}
Obviously, $D = D_{1} + D_{2}$. It's easy to show that $D_{1} \in Der(\mathcal{B})$ and $D_{2} \in Der(\mathcal{C})$. Therefore, $Der(\mathcal{A}) = Der(\mathcal{B}) \oplus Der(\mathcal{C})$.

(b),(c),(d),(e). Similarly to the proof of (a).
\end{proof}
\begin{thm}\label{thm:4.6}
Let $\mathcal{A}$ be a ternary Jordan algebra over $\rm{F}$. Then $ZDer(\mathcal{A}) = \Gamma(\mathcal{A}) \cap Der(\mathcal{A})$.
\end{thm}
\begin{proof}
Assume that $f \in \Gamma(\mathcal{A}) \cap Der(\mathcal{A})$, then for any $x, y, z \in \mathcal{A}$, we have
\[\it{f}(\rm{[\![\it{x},\it{y},\it{z}]\!]}) = \rm{[\![\it{f}(\it{x}),\it{y},\it{z}]\!]} + \rm{[\![\it{x},\it{f}(\it{y}),\it{z}]\!]} + \rm{[\![\it{x},\it{y},\it{f}(\it{z})]\!]},\]
\[\it{f}(\rm{[\![\it{x},\it{y},\it{z}]\!]}) = \rm{[\![\it{f}(\it{x}),\it{y},\it{z}]\!]} = \rm{[\![\it{x},\it{f}(\it{y}),\it{z}]\!]} = \rm{[\![\it{x},\it{y},\it{f}(\it{z})]\!]},\]
hence we have $\it{f}(\rm{[\![\it{x},\it{y},\it{z}]\!]}) = 0$, i.e., $\it{f}(\rm{[\![\mathcal{A},\mathcal{A},\mathcal{A}]\!]}) = 0$.
\begin{align*}
&\rm{[\![\it{f}(\it{x}),\it{y},\it{z}]\!]} = \it{f}(\rm{[\![\it{x},\it{y},\it{z}]\!]}) = 0,
\end{align*}
which implies that $f(x) \in Z(\mathcal{A})$, i.e., $f(\mathcal{A}) \subseteq Z(\mathcal{A})$. Therefore, $f \in ZDer(\mathcal{A})$, i.e., $\Gamma(\mathcal{A}) \cap Der(\mathcal{A}) \subseteq ZDer(\mathcal{A})$.

On the other hand, suppose that $f \in ZDer(\mathcal{A})$. Obviously, $f \in Der(\mathcal{A})$. For any $x, y, z \in \mathcal{A}$, we have
\[\it{f}(\rm{[\![\it{x},\it{y},\it{z}]\!]}) = 0,\]
\[\rm{[\![\it{f}(\it{x}),\it{y},\it{z}]\!]} = \rm{[\![\it{x},\it{f}(\it{y}),\it{z}]\!]} = \rm{[\![\it{x},\it{y},\it{f}(\it{z})]\!]} = 0,\]
hence we have $\it{f}(\rm{[\![\it{x},\it{y},\it{z}]\!]}) = \rm{[\![\it{f}(\it{x}),\it{y},\it{z}]\!]} = \rm{[\![\it{x},\it{f}(\it{y}),\it{z}]\!]} = \rm{[\![\it{x},\it{y},\it{f}(\it{z})]\!]}$, which implies that $f \in \Gamma(\mathcal{A})$. Therefore, $f \in \Gamma(\mathcal{A}) \cap Der(\mathcal{A})$, i.e., $ZDer(\mathcal{A}) \subseteq \Gamma(\mathcal{A}) \cap Der(\mathcal{A})$.

Therefore, $ZDer(\mathcal{A}) = \Gamma(\mathcal{A}) \cap Der(\mathcal{A})$.
\end{proof}
\begin{thm}\label{thm:4.7}
Suppose that $\mathcal{A}$ is a ternary Jordan algebra. Then $(Q\Gamma(\mathcal{A}), \bullet)$ is a Jordan algebra with
\[f \bullet g = fg + gf, \quad\forall f, g \in Q\Gamma(\mathcal{A}).\]
\end{thm}
\begin{proof}
We only need to show that $(Q\Gamma(\mathcal{A}), \bullet)$ is a subalgebra of $(\rm{End}(\mathcal{A}), \bullet)$. For any $f, g \in Q\Gamma(\mathcal{A})$ and $x, y, z \in \mathcal{A}$, we have
\begin{align*}
&\rm{[\![\it{f}\bullet\it{g}(\it{x}),\it{y},\it{z}]\!]} = \rm{[\![\it{f}\it{g}(\it{x}),\it{y},\it{z}]\!]} + \rm{[\![\it{g}\it{f}(\it{x}),\it{y},\it{z}]\!]}\\
&= \rm{[\![\it{g}(\it{x}),\it{f}({y}),\it{z}]\!]} + \rm{[\![\it{f}(\it{x}),\it{g}(\it{y}),\it{z}]\!]} = \rm{[\![\it{x},\it{g}\it{f}({y}),\it{z}]\!]} + \rm{[\![\it{x},\it{f}\it{g}(\it{y}),\it{z}]\!]}\\
&= \rm{[\![\it{x},\it{f}\bullet\it{g}(\it{y}),\it{z}]\!]},
\end{align*}
similarly we have $\rm{[\![\it{f}\bullet\it{g}(\it{x}),\it{y},\it{z}]\!]} = \rm{[\![\it{x},\it{y},\it{f}\bullet\it{g}(\it{z})]\!]}$. Hence, $f \bullet g \in Q\Gamma(\mathcal{A})$. Therefore, $(Q\Gamma(\mathcal{A}), \bullet)$ is a Jordan algebra.
\end{proof}
\section{Quasiderivations of ternary Jordan algebras}\label{se:5}
In this section, we'll show that quasiderivations of a ternary Jordan algebra can be embedded as derivations into a lager ternary Jordan algebra. And we also obtain a direct sum decomposition of $Der(\mathcal{\tilde{A}})$ when $Z(\mathcal{A}) = \{0\}$. Suppose that $t$ is an indeterminant. We consider the linear space $\rm{F}[\it{t}]/\langle\it{t}^{4}\rangle$. For convenience, let $t^{i}$ denote the congruence class of $t^{i}$ in $\rm{F}[\it{t}]/\langle\it{t}^{4}\rangle$.
\begin{prop}\label{prop:5.1}
Let $\mathcal{A}$ be a ternary Jordan algebra over $\rm{F}$. We define $\tilde{\mathcal{A}}:= \{\sum(x \otimes t + y \otimes t^{2} + z \otimes t^{3}) \mid x, y, z \in \mathcal{A}\} \subseteq \mathcal{A} \otimes (\rm{F}[\it{t}]/\langle\it{t}^{4}\rangle)$ and $\rm{[\![\it{x \otimes t^{i}},\it{y \otimes t^{j}},\it{z \otimes t^{k}}]\!]} = \rm{[\![\it{x},\it{y},\it{z}]\!]} \otimes \it{t^{i + j + k}}$ where $x, y, z \in \mathcal{A}$ and $i, j, k \in \{1, 2, 3\}$. Then $(\tilde{\mathcal{A}}, \rm{[\![\cdot,\cdot,\cdot]\!]})$ is a ternary Jordan algebra.
\end{prop}
\begin{proof}
It's obvious that $\rm{[\![\cdot,\cdot,\cdot]\!]}$ is symmetric.

For any $x_{1} \otimes t^{i_{1}}, x_{2} \otimes t^{i_{2}}, y_{1} \otimes t^{j_{1}}, y_{2} \otimes t^{j_{2}}, a \otimes t^{\alpha}, b \otimes t^{\beta}, c \otimes t^{\gamma} \in \mathcal{A}$, we have
\begin{align*}
&D_{(x_{1} \otimes t^{i_{1}}, x_{2} \otimes t^{i_{2}}), (y_{1} \otimes t^{j_{1}}, y_{2} \otimes t^{j_{2}})}(\rm{[\![\it{a \otimes t^{\alpha}},\it{b \otimes t^{\beta}},\it{c \otimes t^{\gamma}}]\!]})\\
&= D_{(x_{1} \otimes t^{i_{1}}, x_{2} \otimes t^{i_{2}}), (y_{1} \otimes t^{j_{1}}, y_{2} \otimes t^{j_{2}})}(\rm{[\![\it{a},\it{b},\it{c}]\!]} \otimes \it{t^{\alpha + \beta + \gamma}})\\
&= \rm{[\![\rm{[\![\rm{[\![\it{a},\it{b},\it{c}]\!]} \otimes \it{t^{\alpha + \beta + \gamma}},\it{x_{1} \otimes t^{i_{1}}},\it{x_{2} \otimes t^{i_{2}}}]\!]},\it{y_{1} \otimes t^{j_{1}}}, \it{y_{2} \otimes t^{j_{2}}}]\!]}\\
&- \rm{[\![\rm{[\![\rm{[\![\it{a},\it{b},\it{c}]\!]} \otimes \it{t^{\alpha + \beta + \gamma}},\it{y_{1} \otimes t^{j_{1}}},\it{y_{2} \otimes t^{j_{2}}}]\!]},\it{x_{1} \otimes t^{i_{1}}}, \it{x_{2} \otimes t^{i_{2}}}]\!]}\\
&= (\rm{[\![\rm{[\![\rm{[\![\it{a},\it{b},\it{c}]\!]},\it{x_{1}},\it{x_{2}}]\!]},\it{y_{1}}, \it{y_{2}}]\!]} - \rm{[\![\rm{[\![\rm{[\![\it{a},\it{b},\it{c}]\!]},\it{y_{1}},\it{y_{2}}]\!]},\it{x_{1}}, \it{x_{2}}]\!]}) \otimes \it{t^{\alpha + \beta + \gamma + i_{1} + i_{2} + j_{1} + j_{2}}}\\
&= D_{(x_{1}, x_{2}), (y_{1}, y_{2})}(\rm{[\![\it{a},\it{b},\it{c}]\!]}) \otimes \it{t^{\alpha + \beta + \gamma + i_{1} + i_{2} + j_{1} + j_{2}}}\\
&= (\rm{[\![\it{D_{(x_{1}, x_{2}), (y_{1}, y_{2})}}(\it{a}),\it{b},\it{c}]\!]} + \rm{[\![\it{a},D_{(x_{1}, x_{2}), (y_{1}, y_{2})}(\it{b}),\it{c}]\!]} + \rm{[\![\it{a},\it{b},D_{(x_{1}, x_{2}), (y_{1}, y_{2})}(\it{c})]\!]})\\
&\otimes \it{t^{\alpha + \beta + \gamma + i_{1} + i_{2} + j_{1} + j_{2}}}\\
&= \rm{[\![\it{D_{(x_{1} \otimes t^{i_{1}}, x_{2} \otimes t^{i_{2}}), (y_{1} \otimes t^{j_{1}}, y_{2} \otimes t^{j_{2}})}}(\it{a \otimes t^{\alpha}}),\it{b \otimes t^{\beta}},\it{c \otimes t^{\gamma}}]\!]}\\
&+ \rm{[\![\it{a \otimes t^{\alpha}},D_{(x_{1} \otimes t^{i_{1}}, x_{2} \otimes t^{i_{2}}), (y_{1} \otimes t^{j_{1}}, y_{2} \otimes t^{j_{2}})}(\it{b \otimes t^{\beta}}),\it{c \otimes t^{\gamma}}]\!]}\\
&+ \rm{[\![\it{a \otimes t^{\alpha}},\it{b \otimes t^{\beta}},D_{(x_{1} \otimes t^{i_{1}}, x_{2} \otimes t^{i_{2}}), (y_{1} \otimes t^{j_{1}}, y_{2} \otimes t^{j_{2}})}(\it{c \otimes t^{\gamma}})]\!]},
\end{align*}
which implies that $D_{(x_{1} \otimes t^{i_{1}}, x_{2} \otimes t^{i_{2}}), (y_{1} \otimes t^{j_{1}}, y_{2} \otimes t^{j_{2}})} \in Der(\tilde{\mathcal{A}})$. Therefore, $(\tilde{\mathcal{A}}, \rm{[\![\cdot,\cdot,\cdot]\!]})$ is a ternary Jordan algebra.
\end{proof}

For convenience, we write $xt, xt^{2}, xt^{3}$ in place of $x \otimes t, x \otimes t^{2}, x \otimes t^{3}$. Then the multiplication of $\tilde{\mathcal{A}}$ is
\[\rm{[\![\it{xt},\it{yt},\it{zt}]\!]} = \rm{[\![\it{x},\it{y},\it{z}]\!]}\it{t^{3}}, \quad\forall x, y, z \in \mathcal{A},\]
and $\rm{[\![\it{xt^{i}},\it{yt^{j}},\it{zt^{k}}]\!]} = 0$ in the case at least one of $i, j, k$ is larger than $1$.

Let $U$ be a subspace of $\mathcal{A}$ satisfying $\mathcal{A} = U \oplus \rm{[\![\mathcal{A},\mathcal{A},\mathcal{A}]\!]}$. Then
\[\tilde{\mathcal{A}} = \mathcal{A}t + \mathcal{A}t^{2} + \mathcal{A}t^{3} = \mathcal{A}t + \mathcal{A}t^{2} + \rm{[\![\mathcal{A},\mathcal{A},\mathcal{A}]\!]}\it{t}^{3} + U\it{t}^{3}.\]

Define a linear map $l_{u} : QDer(\mathcal{A}) \rightarrow \rm{End}(\tilde{\mathcal{A}})$ by
\[l_{u}(f)(at + bt^{2} + ct^{3} + ut^{3}) = f(a)t + f^{'}(c)t^{3}\]
where $(f,f,f,f^{'}) \in \Delta(\mathcal{A})$ and $a, b \in \mathcal{A}, c \in \rm{[\![\mathcal{A},\mathcal{A},\mathcal{A}]\!]}, \it{u \in U}$.
\begin{thm}\label{thm:5.2}
$\mathcal{A}, \tilde{\mathcal{A}}, l_{u}$ are defines as above. Then
\begin{enumerate}[(1)]
\item $l_{u}$ is injective and does not depend on the choice of $f^{'}$;
\item $l_{u}(QDer(\mathcal{A})) \subseteq Der(\tilde{\mathcal{A}})$.
\end{enumerate}
\end{thm}
\begin{proof}
(1). Suppose there exist $f, g$ such that $l_{u}(f) = l_{u}(g)$. Then we have
\[l_{u}(f)(at + bt^{2} + ct^{3} + ut^{3}) = l_{u}(g)(at + bt^{2} + ct^{3} + ut^{3}),\quad\forall a, b \in \mathcal{A}, c \in \rm{[\![\mathcal{A},\mathcal{A},\mathcal{A}]\!]}, \it{u \in U},\]
which implies that
\[f(a)t + f^{'}(c)t^{3} = g(a)t + g^{'}(c)t^{3}.\]
Hence we have $f(a) = g(a)$, which implies that $f = g$. Therefore, $l_{u}$ is injective.

Suppose that there exists $f^{''}$ such that $l_{u}(f)(at + bt^{2} + ct^{3} + ut^{3}) = f(a)t + f^{''}(c)t^{3}$ and $\rm{[\![\it{f}(\it{x}), \it{y}, \it{z}]\!]} + \rm{[\![\it{x}, \it{f}(\it{y}), \it{z}]\!]} + \rm{[\![\it{x}, \it{y}, \it{f}(\it{z})]\!]} = \it{f}^{''}(\rm{[\![\it{x}, \it{y}, \it{z}]\!]})$, then we have $\it{f}^{'}(\rm{[\![\it{x}, \it{y}, \it{z}]\!]}) = \it{f}^{''}(\rm{[\![\it{x}, \it{y}, \it{z}]\!]})$, hence we have $\it{f}^{'}(c) = \it{f}^{''}(c)$.

Hence, we have
\[l_{u}(f)(at + bt^{2} + ct^{3} + ut^{3}) = f(a)t + f^{'}(c)t^{3} = f(a)t + f^{''}(c)t^{3},\]
which implies that $l_{u}$ does not depend on the choice of $f^{'}$.

(2). For all $a_{i}t + b_{i}t^{2} + c_{i}t^{3} + u_{i}t^{3} \in \tilde{\mathcal{A}}$, $i = 1, 2, 3$, we have
\begin{align*}
&\rm{[\![\it{l_{u}(f)}(\it{a_{1}t + b_{1}t^{2} + c_{1}t^{3} + u_{1}t^{3}}), \it{a_{2}t + b_{2}t^{2} + c_{2}t^{3} + u_{2}t^{3}}, \it{a_{3}t + b_{3}t^{2} + c_{3}t^{3} + u_{3}t^{3}}]\!]}\\
&+ \rm{[\![\it{a_{1}t + b_{1}t^{2} + c_{1}t^{3} + u_{1}t^{3}}, \it{l_{u}(f)}(\it{a_{2}t + b_{2}t^{2} + c_{2}t^{3}) + u_{2}t^{3}}), \it{a_{3}t + b_{3}t^{2} + c_{3}t^{3} + u_{3}t^{3}}]\!]}\\
&+ \rm{[\![\it{a_{1}t + b_{1}t^{2} + c_{1}t^{3} + u_{1}t^{3}}, \it{a_{2}t + b_{2}t^{2} + c_{2}t^{3} + u_{2}t^{3}}, \it{l_{u}(f)}(\it{a_{3}t + b_{3}t^{2} + c_{3}t^{3} + u_{3}t^{3}})]\!]}\\
&= \rm{[\![\it{f}(\it{a_{1}}), \it{a_{2}}, \it{a_{3}}]\!]}\it{t}^{3} + \rm{[\![\it{a_{1}}, \it{f}(\it{a_{2}}), \it{a_{3}}]\!]}\it{t}^{3} + \rm{[\![\it{a_{1}}, \it{a_{2}}, \it{f}(\it{a_{3}})]\!]}\it{t}^{3}\\
&= \it{f}^{'}(\rm{[\![\it{a_{1}}, \it{a_{2}}, \it{a_{3}}]\!]})\it{t}^{3},
\end{align*}
\begin{align*}
&\it{l_{u}(f)}(\rm{[\![\it{a_{1}t + b_{1}t^{2} + c_{1}t^{3} + u_{1}t^{3}}, \it{a_{2}t + b_{2}t^{2} + c_{2}t^{3} + u_{2}t^{3}}, \it{a_{3}t + b_{3}t^{2} + c_{3}t^{3} + u_{3}t^{3}}]\!]})\\
&= \it{l_{u}(f)}(\rm{[\![\it{a_{1}}, \it{a_{2}}, \it{a_{3}}]\!]}\it{t}^{3}) = \it{f}^{'}(\rm{[\![\it{a_{1}}, \it{a_{2}}, \it{a_{3}}]\!]})\it{t}^{3}.
\end{align*}
Therefore, $l_{u}(f) \in Der(\tilde{\mathcal{A}})$. It follows that $l_{u}(QDer(\mathcal{A})) \subseteq Der(\tilde{\mathcal{A}})$.
\end{proof}
\begin{thm}\label{thm:5.3}
Let $\mathcal{A}$ be a ternary Jordan algebra. $\tilde{\mathcal{A}}, l_{u}$ are defined as above. If $Z(\mathcal{A}) = \{0\}$, then $Der(\tilde{\mathcal{A}}) = l_{u}(QDer(\mathcal{A})) \dotplus ZDer(\tilde{\mathcal{A}})$.
\end{thm}
\begin{proof}
Since $Z(\mathcal{A}) = \{0\}$, we have $Z(\tilde{\mathcal{A}}) = \mathcal{A}t^{2} + \mathcal{A}t^{3}$. Moreover, for all $g \in Der(\tilde{\mathcal{A}})$, we have $g(Z(\tilde{\mathcal{A}})) \subseteq Z(\tilde{\mathcal{A}})$.

Define a map $f : \tilde{\mathcal{A}} \rightarrow \mathcal{A}t^{2} + \mathcal{A}t^{3}$ by
$$f(x) =
\left\{
\begin{aligned}
g(x) \cap (\mathcal{A}t^{2} + \mathcal{A}t^{3}),\quad x \in \mathcal{A}t\\
g(x),\quad x \in \mathcal{A}t^{2} + Ut^{3}\\
0,\quad x \in \rm{[\![\mathcal{A},\mathcal{A},\mathcal{A}]\!]}\it{t}^{3}
\end{aligned}
\right.
$$
Obviously, $f$ is linear. Moreover,
\[\it{f}(\rm{[\![\tilde{\mathcal{A}},\tilde{\mathcal{A}},\tilde{\mathcal{A}}]\!]}) = \it{f}(\rm{[\![\mathcal{A},\mathcal{A},\mathcal{A}]\!]}\it{t}^{3}) = \rm{0},\]
by the definition of $f$, we see that $f(\tilde{\mathcal{A}}) \subseteq Z(\tilde{\mathcal{A}})$. Therefore, $f$ is an element of $ZDer(\tilde{\mathcal{A}})$.

Since
\[(g - f)(\mathcal{A}t) = g(\mathcal{A}t) - f(\mathcal{A}t) = g(\mathcal{A}t) - g(\mathcal{A}t) \cap (\mathcal{A}t^{2} + \mathcal{A}t^{3}) \subseteq \mathcal{A}t,\]
\[(g - f)(\mathcal{A}t^{2} + Ut^{3}) = g(\mathcal{A}t^{2} + Ut^{3}) - f(\mathcal{A}t^{2} + Ut^{3}) = 0,\]
\[(g - f)(\rm{[\![\mathcal{A},\mathcal{A},\mathcal{A}]\!]}\it{t}^{3}) = g(\rm{[\![\mathcal{A},\mathcal{A},\mathcal{A}]\!]}\it{t}^{3}) = g(\rm{[\![\tilde{\mathcal{A}},\tilde{\mathcal{A}},\tilde{\mathcal{A}}]\!]}) \subseteq \rm{[\![\tilde{\mathcal{A}},\tilde{\mathcal{A}},\tilde{\mathcal{A}}]\!]} = \rm{[\![\mathcal{A},\mathcal{A},\mathcal{A}]\!]}\it{t}^{3},\]
hence there exist linear maps $h, h^{'} \in \rm{End}(\mathcal{A})$ such that for all $a \in \mathcal{A}, c \in \rm{[\![\mathcal{A},\mathcal{A},\mathcal{A}]\!]}$,
\[(g - f)(at) = h(a)t, (g - f)(ct^{3}) = h^{'}(c)t^{3}.\]
Note that $g - f \in Der(\tilde{\mathcal{A}})$, we have
\[\it{(g - f)}(\rm{[\![\it{at},\it{bt},\it{ct}]\!]}) = \rm{[\![\it{(g - f)}(\it{at}),\it{bt},\it{ct}]\!]} + \rm{[\![\it{at},\it{(g - f)}(\it{bt}),\it{ct}]\!]} + \rm{[\![\it{at},\it{bt},\it{(g - f)}(\it{ct})]\!]},\]
which implies that
\[\it{h}^{'}(\rm{[\![\it{a},\it{b},\it{c}]\!]})\it{t}^{3} = \rm{[\![\it{h}(\it{a}),\it{b},\it{c}]\!]}\it{t}^{3} + \rm{[\![\it{a},\it{h}(\it{b}),\it{c}]\!]}\it{t}^{3} + \rm{[\![\it{a},\it{b},\it{h}(\it{c})]\!]}\it{t}^{3}.\]
Therefore, we see that $h \in QDer(\mathcal{A})$. So $g - f = l_{u}(h) \in l_{u}(QDer(\mathcal{A}))$. According to Theorem \ref{thm:5.2} (2), we have $Der(\tilde{\mathcal{A}}) = ZDer(\tilde{\mathcal{A}}) + l_{u}(QDer(\mathcal{A}))$.

If there exists $f \in ZDer(\tilde{\mathcal{A}}) \cap l_{u}(QDer(\mathcal{A}))$, then there exists $h \in QDer(\mathcal{A})$ such that $f = l_{u}(h)$. So
\[f(at + bt^{2} + ct^{3} + ut^{3}) = h(a)t + h^{'}(c)t^{3}.\]

On the other hand, $f(at + bt^{2} + ct^{3} + ut^{3}) \subseteq \mathcal{A}t^{2} + \mathcal{A}t^{3}$ since $f \in ZDer(\tilde{\mathcal{A}})$. So $h(a) = 0$, i.e., $h = 0$. Therefore, $f = 0$.

Hence, $Der(\tilde{\mathcal{A}}) = l_{u}(QDer(\mathcal{A})) \dotplus ZDer(\tilde{\mathcal{A}})$.
\end{proof}
\section{Ternary Jordan algebras with $QDer(\mathcal{A}) = gl(\mathcal{A})$}\label{se:6}
Let $\mathcal{A}$ be a ternary Jordan algebra over $\rm{F}$ with $char \rm{F} \neq 2$. We define a linear map $\phi : \mathcal{A} \otimes \mathcal{A} \otimes \mathcal{A} \rightarrow \mathcal{A}, x \otimes y \otimes z \mapsto \rm{[\![\it{x},\it{y},\it{z}]\!]}$. Define $Ker(\phi) = \{\sum x \otimes y \otimes z \in \mathcal{A} \otimes \mathcal{A} \otimes \mathcal{A} \mid \sum\rm{[\![\it{x},\it{y},\it{z}]\!]} = 0, \it{x}, \it{y}, \it{z} \in \mathcal{A}\}$, then it is easy to see that $Ker(\phi)$ is a subspace of $\mathcal{A} \otimes \mathcal{A} \otimes \mathcal{A}$.

We define $(\mathcal{A} \otimes \mathcal{A} \otimes \mathcal{A})^{+} = \{x \otimes y \otimes z + y \otimes x \otimes z \mid x, y, z \in \mathcal{A}\}$ and $(\mathcal{A} \otimes \mathcal{A} \otimes \mathcal{A})^{-} = \{x \otimes y \otimes z - y \otimes x \otimes z \mid x, y, z \in \mathcal{A}\}$, then both $(\mathcal{A} \otimes \mathcal{A} \otimes \mathcal{A})^{+}$ and $(\mathcal{A} \otimes \mathcal{A} \otimes \mathcal{A})^{-}$ are subspaces of $\mathcal{A} \otimes \mathcal{A} \otimes \mathcal{A}$. It is easy to check that
\[\mathcal{A} \otimes \mathcal{A} \otimes \mathcal{A} = (\mathcal{A} \otimes \mathcal{A} \otimes \mathcal{A})^{+} \dotplus (\mathcal{A} \otimes \mathcal{A} \otimes \mathcal{A})^{-}.\]
and we also have $\rm{dim}(\mathcal{A} \otimes \mathcal{A} \otimes \mathcal{A})^{+} = \it{n^{2}(n + 1)/2}$ and $\rm{dim}(\mathcal{A} \otimes \mathcal{A} \otimes \mathcal{A})^{-} = \it{n^{2}(n - 1)/2}$, where $\rm{dim}\mathcal{A} = \it{n}$.

For all $D \in gl(\mathcal{A})$, we define $D^{*} \in gl(\mathcal{A} \otimes \mathcal{A} \otimes \mathcal{A})$ satisfying that
\begin{equation}\label{eq:6.1}
D^{*}(x \otimes y \otimes z) = D(x) \otimes y \otimes z + x \otimes D(y) \otimes z + x \otimes y \otimes D(z)\tag{6.1}
\end{equation}
for all $x, y, z \in \mathcal{A}$.
\begin{thm}\label{thm:6.1}
$D \in QDer(\mathcal{A})$ if and only if $D^{*}(Ker(\phi)) \subseteq Ker(\phi)$.
\end{thm}
\begin{proof}
For all $\sum x \otimes y \otimes z \in Ker(\phi)$, we have
\[D^{*}(\sum x \otimes y \otimes z) = \sum D^{*}(x \otimes y \otimes z) = \sum (D(x) \otimes y \otimes z + x \otimes D(y) \otimes z + x \otimes y \otimes D(z)),\]
since $D \in QDer(\mathcal{A})$, we have
\[\sum (\rm{[\![\it{D(x)}, \it{y}, \it{z}]\!]} + \rm{[\![\it{x}, \it{D(y)}, \it{z}]\!]} + \rm{[\![\it{x}, \it{y}, \it{D(z)}]\!]}) = \sum \it{D}^{'}(\rm{[\![\it{x}, \it{y}, \it{z}]\!]}) = \it{D}^{'}(\sum\rm{[\![\it{x}, \it{y}, \it{z}]\!]}) = 0.\]
Hence, we have $D^{*}(Ker(\phi)) \subseteq Ker(\phi)$.

Conversely, let $U$ be a complement of $\rm{[\![\mathcal{A}, \mathcal{A}, \mathcal{A}]\!]}$ in $\mathcal{A}$, that is $\mathcal{A} = \rm{[\![\mathcal{A}, \mathcal{A}, \mathcal{A}]\!]} \oplus \it{U}$. For $D \in gl(\mathcal{A})$ satisfying $D^{*}(Ker(\phi)) \subseteq Ker(\phi)$, define a linear map $D^{'} : \mathcal{A} \rightarrow \mathcal{A}$ by
$$D^{'}(z) =
\left\{
\begin{aligned}
0,\quad z\in U,\\
\sum_{i = 1}^{m}\phi(D^{*}(x_{i} \otimes y_{i} \otimes z_{i})),\quad z = \sum_{i = 1}^{m}\phi(x_{i} \otimes y_{i} \otimes z_{i}) \in \rm{[\![\mathcal{A}, \mathcal{A}, \mathcal{A}]\!]}.
\end{aligned}
\right.
$$
It follows from a direct computation that $D^{'}$ is well-defined. And we see that
\[\phi(D(x) \otimes y \otimes z) + \phi(x \otimes D(y) \otimes z) + \phi(x \otimes y \otimes D(z)) = \phi(D^{*}(x \otimes y \otimes z)) = D^{'}(\phi(x \otimes y \otimes z)).\]
It follows $D \in QDer(\mathcal{A})$.
\end{proof}
\begin{thm}\label{thm:6.2}
Suppose that $gl(\mathcal{A})$ acts on $\mathcal{A} \otimes \mathcal{A} \otimes \mathcal{A}$ via $D \cdot (x \otimes y \otimes z) = D^{*}(x \otimes y \otimes z)$. Then $(\mathcal{A} \otimes \mathcal{A} \otimes \mathcal{A})^{+}$ and $(\mathcal{A} \otimes \mathcal{A} \otimes \mathcal{A})^{-}$ are two irreducible $gl(\mathcal{A})$-modules.
\end{thm}
\begin{proof}
For any $D \in gl(\mathcal{A})$ and $x \otimes y \otimes z + y \otimes x \otimes z \in (\mathcal{A} \otimes \mathcal{A} \otimes \mathcal{A})^{+}$, we have
\begin{align*}
&D \cdot (x \otimes y \otimes z + y \otimes x \otimes z) = D^{*}(x \otimes y \otimes z + y \otimes x \otimes z)\\
&= D(x) \otimes y \otimes z + D(y) \otimes x \otimes z + x \otimes D(y) \otimes z + y \otimes D(x) \otimes z + x \otimes y \otimes D(z)\\
&+ y \otimes x \otimes D(z)\\
&= (D(x) \otimes y \otimes z + y \otimes D(x) \otimes z) + (D(y) \otimes x \otimes z + x \otimes D(y) \otimes z)\\
&+ (x \otimes y \otimes D(z) + y \otimes x \otimes D(z)) \in (\mathcal{A} \otimes \mathcal{A} \otimes \mathcal{A})^{+}.
\end{align*}
It follows from a direct computation that
\[[D_{1}, D_{2}] \cdot (x \otimes y \otimes z) = D_{1} \cdot (D_{2} \cdot (x \otimes y \otimes z)) - D_{2} \cdot (D_{1} \cdot (x \otimes y \otimes z)),\quad\forall D_{1}, D_{2} \in gl(\mathcal{A}).\]
Hence, $(\mathcal{A} \otimes \mathcal{A} \otimes \mathcal{A})^{+}$ is a $gl(\mathcal{A})$-module.

Suppose that there exists a nonzero submodule $V$ of $(\mathcal{A} \otimes \mathcal{A} \otimes \mathcal{A})^{+}$. Choose a nonzero element $\sum x \otimes y \otimes z + y \otimes x \otimes z \in V$ for some $x, y, z \in \mathcal{A}$. A direct computation shows that all elements in $(\mathcal{A} \otimes \mathcal{A} \otimes \mathcal{A})^{+}$ are obtained by repeated application of elements of $gl(\mathcal{A})$ to $\sum x \otimes y \otimes z + y \otimes x \otimes z$ and formation of linear combinations. Hence $V$ is $(\mathcal{A} \otimes \mathcal{A} \otimes \mathcal{A})^{+}$ itself. Thus, $(\mathcal{A} \otimes \mathcal{A} \otimes \mathcal{A})^{+}$ as a $gl(\mathcal{A})$-module is irreducible. Similarly, $(\mathcal{A} \otimes \mathcal{A} \otimes \mathcal{A})^{-}$ is also an irreducible $gl(\mathcal{A})$-module.
\end{proof}
\begin{thm}\label{thm:6.3}
Let $\mathcal{A}$ be a ternary Jordan algebra with $\rm{[\![\mathcal{A}, \mathcal{A}, \mathcal{A}]\!]} \neq 0$ and $QDer(\mathcal{A}) = gl(\mathcal{A})$. Then $\rm{dim}(\mathcal{A}) \leq 2$. Moreover,
\begin{enumerate}[(1)]
\item if $\rm{dim}(\mathcal{A}) = 1$, then $\mathcal{A}$ is a simple ternary Jordan algebra;
\item if $\rm{dim}(\mathcal{A}) = 2$, then $\rm{[\![\mathcal{A}, \mathcal{A}, \mathcal{A}]\!]} = \mathcal{A}$.
\end{enumerate}
\end{thm}
\begin{proof}
We consider the action of $gl(\mathcal{A})$ on $\mathcal{A} \otimes \mathcal{A} \otimes \mathcal{A}$ via $D \cdot (x \otimes y \otimes z) = D^{*}(x \otimes y \otimes z)$ for all $x, y, z \in \mathcal{A}$ with $D^{*}$ as in (\ref{eq:6.1}). By Theorem \ref{thm:6.1}, $QDer(\mathcal{A}) = gl(\mathcal{A})$ implies that $gl(\mathcal{A}) \cdot Ker(\phi) \subseteq Ker(\phi)$. Theorem \ref{thm:6.2} tells us that the only proper subspaces of $\mathcal{A} \otimes \mathcal{A} \otimes \mathcal{A}$, invariant under this action of $gl(\mathcal{A})$, are $(\mathcal{A} \otimes \mathcal{A} \otimes \mathcal{A})^{+}$ and $(\mathcal{A} \otimes \mathcal{A} \otimes \mathcal{A})^{-}$. Thus we have $\phi : \mathcal{A} \otimes \mathcal{A} \otimes \mathcal{A} \rightarrow \mathcal{A}$ with kernel $\{0\}$, $(\mathcal{A} \otimes \mathcal{A} \otimes \mathcal{A})^{+}$ and $(\mathcal{A} \otimes \mathcal{A} \otimes \mathcal{A})^{-}$. Using
\[\rm{dim}(\mathcal{A}) \geq \rm{dim}(\mathcal{A} \otimes \mathcal{A} \otimes \mathcal{A}) - \rm{dim}(Ker(\phi)),\]
we have $\rm{dim}(\mathcal{A}) = 1$ if $Ker(\phi) = \{0\}$ or $(\mathcal{A} \otimes \mathcal{A} \otimes \mathcal{A})^{-}$ and $\rm{dim}(\mathcal{A}) \leq 2$ if $Ker(\phi) = (\mathcal{A} \otimes \mathcal{A} \otimes \mathcal{A})^{+}$.
\begin{enumerate}[(1)]
\item If $\rm{dim}(\mathcal{A}) = 1$, then $\mathcal{A}$ is simple since $\rm{[\![\mathcal{A}, \mathcal{A}, \mathcal{A}]\!]} \neq 0$.
\item If $\rm{dim}(\mathcal{A}) = 2$, i.e., $Ker(\phi) = (\mathcal{A} \otimes \mathcal{A} \otimes \mathcal{A})^{+}$, hence $\rm{dim}(\it{Ker}(\phi)) = \rm{6}$ and $\rm{dim}(\rm{[\![\mathcal{A}, \mathcal{A}, \mathcal{A}]\!]}) = 2$, so $\phi$ must be surjective and we have $\rm{[\![\mathcal{A}, \mathcal{A}, \mathcal{A}]\!]} = \mathcal{A}$.
\end{enumerate}
\end{proof}
\section{Centroids of ternary Jordan algebras}\label{se:7}
\begin{prop}\label{prop:7.1}
If $\mathcal{A}$ has no nonzero ideals $\mathcal{I}$, $\mathcal{J}$ with $\rm{[\![\mathcal{A},\mathcal{I},\mathcal{J}]\!]} = 0$, then $\Gamma(\mathcal{A})$ is an integral domain.
\end{prop}
\begin{proof}
First, $id \in \Gamma(\mathcal{A})$. If there exist $\psi \neq 0$, $\varphi \neq 0$ such that $\psi\varphi = 0$, then there exist $x, y \in \mathcal{A}$ such that $\psi(x) \neq 0$, $\varphi(y) \neq 0$. Then $\rm{[\![\mathcal{A},\it{\psi(x)},\it{\varphi(y)}]\!]} = \psi\varphi(\rm{[\![\mathcal{A},\it{x},\it{y}]\!]}) = 0$. Therefore $\psi(x)$ and $\varphi(y)$ can span to be two nonzero ideals $\mathcal{I}$, $\mathcal{J}$ of $\mathcal{A}$ such that $\rm{[\![\mathcal{A},\mathcal{I},\mathcal{J}]\!]} = 0$ respectively, a contradiction. Hence, $\Gamma(\mathcal{A})$ has no zero divisor, it is an integral domain.
\end{proof}
\begin{thm}\label{thm:7.2}
If $\mathcal{A}$ is a simple ternary Jordan algebra over an algebraically closed field $\rm{F}$, then $\Gamma(\mathcal{A}) = \rm{F}\it{id}$.
\end{thm}
\begin{proof}
Let $\psi \in \Gamma(\mathcal{A}) \subseteq \rm{End_{F}}(\mathcal{A})$. Since $\rm{F}$ is algebraically closed, $\psi$ has an eigenvalue $\lambda$. We denote the corresponding eigenspace by $E_{\lambda}(\psi)$. So $E_{\lambda}(\psi) \neq \{0\}$. For any $v \in E_{\lambda}(\psi)$, $x, y \in \mathcal{A}$, we have $\psi(\rm{[\![\it{v},\it{x},\it{y}]\!]}) = \rm{[\![\psi(\it{v}),\it{x},\it{y}]\!]} = \lambda\rm{[\![\it{v},\it{x},\it{y}]\!]}$, so $\rm{[\![\it{v},\it{x},\it{y}]\!]} \in E_{\lambda}(\psi)$. It follows that $E_{\lambda}(\psi)$ is an ideal of $\mathcal{A}$. But $\mathcal{A}$ is simple, so $E_{\lambda}(\psi) = \mathcal{A}$, i.e., $\psi = \lambda id_{\mathcal{A}}$. Therefore, $\Gamma(\mathcal{A}) = \rm{F}\it{id}$.
\end{proof}
\begin{prop}\label{prop:7.3}
Let $\mathcal{A}$ be a ternary Jordan algebra over a field $\rm{F}$. Then the following results hold:
\begin{enumerate}[(1)]
\item $\mathcal{A}$ is indecomposable if and only if $\Gamma(\mathcal{A})$ does not contain idempotents except $0$ and id.
\item If $\mathcal{A}$ is perfect, then every $\psi \in \Gamma(\mathcal{A})$ is symmetric with respect to any invariant form on $\mathcal{A}$.
\end{enumerate}
\end{prop}
\begin{proof}
(1). If there exists $\psi \in \Gamma(\mathcal{A})$ which is an idempotent and satisfies $\psi \neq 0, id$, then $\psi^{2}(x) = \psi(x)$ for any $x \in \mathcal{A}$. We can see that $Ker(\psi)$ and $Im(\psi)$ are ideals of $\mathcal{A}$. Moreover, $Ker(\psi) \cap Im(\psi) = 0$. Indeed, if $x \in Ker(\psi) \cap Im(\psi)$, then there exists $y \in \mathcal{A}$ such that $x = \psi(y)$ and $x = \psi(y) = \psi^{2}(y) = \psi(x) = 0$. For any $x \in \mathcal{A}$, we have a decomposition $x = x - \psi(x) + \psi(x)$ where $x - \psi(x) \in Ker(\psi)$ and $\psi(x) \in Im(\psi)$. So we have $\mathcal{A} = Ker(\psi) \oplus Im(\mathcal{A})$, a contradiction.

On the other hand, suppose that $\mathcal{A}$ has a decomposition $\mathcal{A} = \mathcal{A}_{1} \oplus \mathcal{A}_{2}$. Then for any $x \in \mathcal{A}$, we have $x = x_{1} + x_{2}$ where $x_{1} \in \mathcal{A}_{1}$ and $x_{2} \in \mathcal{A}_{2}$. Define $\psi : \mathcal{A} \rightarrow \mathcal{A}$ to be a linear map by $\psi(x) = x_{1} - x_{2}$ where $x = x_{1} + x_{2}$. It's obvious that $\psi^{2} = \psi$. One can verify that $\psi$ satisfies
\[\psi(\rm{[\![\it{x},\it{y},\it{z}]\!]}) = \rm{[\![\psi(\it{x}),\it{y},\it{z}]\!]} = \rm{[\![\it{x},\psi(\it{y}),\it{z}]\!]} = \rm{[\![\it{x},\it{y},\psi(\it{z})]\!]},\quad\forall \it{x, y, z} \in \mathcal{A},\]
which implies that $\psi \in \Gamma(\mathcal{A})$. Contradiction.

(2). Let $f$ be an invariant $\rm{F}$-bilinear form on $\mathcal{A}$. Then $f(\rm{[\![\it{a},\it{b},\it{c}]\!]}, \it{d}) = f(\it{a}, \rm{[\![\it{d},\it{b},\it{c}]\!]})$ for any $\it{a, b, c, d} \in \mathcal{A}$. Let $\psi \in \Gamma(\mathcal{A})$, we have
\begin{align*}
&f(\psi(\rm{[\![\it{a},\it{b},\it{c}]\!]}), \it{d}) = f(\rm{[\![\it{a},\psi(\it{b}),\it{c}]\!]}, \it{d}) = f(\it{a}, \rm{[\![\it{d},\psi(\it{b}),\it{c}]\!]})\\
&= f(\it{a}, \psi(\rm{[\![\it{d},\it{b},\it{c}]\!]})) = \it{f}(\it{a}, \rm{[\![\psi(\it{d}),\it{b},\it{c}]\!]}) = \it{f}(\rm{[\![\it{a},\it{b},\it{c}]\!]}, \psi(\it{d})).
\end{align*}
The lemma is proved since $\mathcal{A}$ is perfect.
\end{proof}
\begin{prop}\label{prop:7.4}
Let $\mathcal{A}$ be a ternary Jordan algebra over $\rm{F}$ and $\mathcal{I}$ be a subset of $\mathcal{A}$. Then $C_{\mathcal{A}}(\mathcal{I}): = \{x \in \mathcal{A} \mid \rm{[\![\it{x},\it{y},\it{z}]\!]} = 0,\;\forall \it{y} \in \mathcal{I}, \it{z} \in \mathcal{A}\}$ is invariant under $\Gamma(\mathcal{A})$, so is any perfect ideal of $\mathcal{A}$.
\end{prop}
\begin{proof}
For any $\psi \in \Gamma(\mathcal{A})$ and $x \in C_{\mathcal{A}}(\mathcal{I})$, we have $\rm{[\![\psi(\it{x}),\it{y},\it{z}]\!]} = \psi(\rm{[\![\it{x},\it{y},\it{z}]\!]}) = 0$ for any $y \in \mathcal{I}, z \in \mathcal{A}$, which implies that $\psi(x) \in C_{\mathcal{A}}(\mathcal{I})$. So $C_{\mathcal{A}}(\mathcal{I})$ is invariant under $\Gamma(\mathcal{A})$.

Suppose that $\mathcal{J}$ is a perfect ideal of $\mathcal{A}$, then $\mathcal{J} = \rm{[\![\mathcal{J},\mathcal{J},\mathcal{J}]\!]}$. For any $y \in \mathcal{J}$, there exist $a, b, c \in \mathcal{J}$ such that $\it{y} = \rm{[\![\it{a},\it{b},\it{c}]\!]}$, then we have $\psi(\it{y}) = \psi(\rm{[\![\it{a},\it{b},\it{c}]\!]}) = \rm{[\![\it{a},\psi(\it{b}),\it{c}]\!]} \in \rm{[\![\mathcal{J},\mathcal{A},\mathcal{A}]\!]} \subseteq \mathcal{J}$. Hence, $\mathcal{J}$ is invariant under $\Gamma(\mathcal{A})$.
\end{proof}
\begin{lem}\label{lem:7.5}\cite{ZCM}
Let $V$ be a linear space and $\varphi : V \rightarrow V$ a linear map. $f(x)$ denotes the minimal polynomial of $\varphi$ . If $x^{2}$ does not divide $f(x)$, then $V = Ker(\varphi) \dotplus Im(\varphi)$.
\end{lem}
\begin{prop}\label{prop:7.6}
Let $\mathcal{A}$ be a ternary Jordan algebra and $f \in \Gamma(\mathcal{A})$. Then the following results hold:
\begin{enumerate}[(1)]
\item $Ker(f)$ and $Im(f)$ are ideals of $\mathcal{A}$;
\item If $\mathcal{A}$ is indecomposable, $f \neq 0$ and $x^{2}$ can't divide the minimal polynomial of $f$, then $f$ is invertible;
\item If $\mathcal{A}$ is indecomposable, $\mathcal{A} = \rm{[\![\mathcal{A},\mathcal{A},\mathcal{A}]\!]}$ and $\Gamma(\mathcal{A})$ consists of semisimple elements, then $\Gamma(\mathcal{A})$ is a field.
\end{enumerate}
\end{prop}
\begin{proof}
(1). It follows from a direct computation.

(2). According to Lemma \ref{lem:7.5}, we have $\mathcal{A} = Ker(f) \oplus Im(f)$. Note that $\mathcal{A}$ is indecomposable and $f \neq 0$, we have $Ker(f) = 0$ and $Im(f) = \mathcal{A}$. Hence, $f$ is invertible.

(3). According to (2), we have $f$ is invertible for any $f \in \Gamma(\mathcal{A})$. It's obvious that $id \in \Gamma(\mathcal{A})$. Suppose that there exist $f_{1}, f_{2} \in \Gamma(\mathcal{A})$ such that $f_{1}f_{2} = 0$. On the other hand, we have $f_{1}f_{2}$ is invertible since $f_{1}, f_{2}$ are invertible. Contraction. Hence, $\Gamma(\mathcal{A})$ has no zero divisor. For any $f_{1}, f_{2} \in \Gamma(\mathcal{A})$, $x, y, z \in \mathcal{A}$, we have
\[\it{f_{1}}\it{f_{2}}(\rm{[\![\it{x},\it{y},\it{z}]\!]}) = \it{f_{1}}(\rm{[\![\it{f_{2}}(\it{x}),\it{y},\it{z}]\!]}) = \rm{[\![\it{f_{2}}(\it{x}),\it{f_{1}}(\it{y}),\it{z}]\!]} = \it{f_{2}}\it{f_{1}}(\rm{[\![\it{x},\it{y},\it{z}]\!]}),\]
which implies that $f_{1}f_{2} = f_{2}f_{1}$ since $\mathcal{A} = \rm{[\![\mathcal{A},\mathcal{A},\mathcal{A}]\!]}$. Therefore, $\Gamma(\mathcal{A})$ is a field.
\end{proof}
\begin{thm}\label{thm:7.7}
Let $\pi : \mathcal{A}_{1} \rightarrow \mathcal{A}_{2}$ be an epimorphism of ternary Jordan algebras. Then for any $f \in \rm{End}_{\rm{F}}(\mathcal{A}_{1}, \it{Ker}({\pi})) = \{\it{g} \in \rm{End}_{\rm{F}}(\mathcal{A}_{1}) \mid \it{g}(Ker(\pi)) \subseteq Ker(\pi)\}$, there exists a unique $\bar{f} \in \rm{End}_{\rm{F}}(\mathcal{A}_{2})$ satisfying $\pi \circ f = \bar{f} \circ \pi$. Moreover, the following results hold:
\begin{enumerate}[(1)]
\item The map $\pi_{\rm{End}} : \rm{End}_{\rm{F}}(\mathcal{A}_{1}, \it{Ker}({\pi})) \rightarrow \rm{End}_{\rm{F}}(\mathcal{A}_{2})$, $\it{f} \mapsto \it{\bar{f}}$ is an algebra homomorphism with the following properties:
    \begin{enumerate}[(a)]
    \item $\pi_{\rm{End}}(Mult(\mathcal{A}_{1})) = Mult(\mathcal{A}_{2})$, $\pi_{\rm{End}}(\Gamma(\mathcal{A}_{1}) \cap \rm{End}_{\rm{F}}(\mathcal{A}_{1}, \it{Ker}({\pi}))) \subseteq \Gamma(\mathcal{A}_{2})$ where $Mult(\mathcal{A}_{i})$ denote the subalgebras of $\rm{End}_{\rm{F}}(\mathcal{A}_{i})$ generated by the right multiplication operators of $\mathcal{A}_{i}$ for $i = 1, 2$.
    \item By restriction, there is an algebra homomorphism
         \[\pi_{\Gamma} : \Gamma(\mathcal{A}_{1}) \cap \rm{End}_{\rm{F}}(\mathcal{A}_{1}, \it{Ker}({\pi})) \rightarrow \Gamma(\mathcal{A}_{2}), f \mapsto \bar{f};\]
    \item If $Ker(\pi) = Z(\mathcal{A}_{1})$, then every $\psi \in \Gamma(\mathcal{A}_{1})$ leaves $Ker(\pi)$ invariant.
    \end{enumerate}
\item Suppose that $\mathcal{A}_{1}$ is perfect and $Ker(\pi) \subseteq Z(\mathcal{A}_{1})$. Then
      \[\pi_{\Gamma} : \Gamma(\mathcal{A}_{1}) \cap \rm{End}_{\rm{F}}(\mathcal{A}_{1}, \it{Ker}({\pi})) \rightarrow \Gamma(\mathcal{A}_{2}), f \mapsto \bar{f}\]
      is injective;
\item If $\mathcal{A}_{1}$ is perfect, $Z(\mathcal{A}_{2}) = \{0\}$ and $Ker(\pi) \subseteq Z(\mathcal{A}_{1})$, then $\pi_{\Gamma} : \Gamma(\mathcal{A}_{1}) \rightarrow \Gamma(\mathcal{A}_{2})$ is an algebra monomorphism.
\end{enumerate}
\end{thm}
\begin{proof}
For any $y \in \mathcal{A}_{2}$, there exists $x \in \mathcal{A}_{1}$ such that $y = \pi(x)$. Define $\bar{f} : \mathcal{A}_{2} \rightarrow \mathcal{A}_{2}$ to be a linear map by $\bar{f}(y) = \pi(f(x))$ where $y = \pi(x)$. We see that $\bar{f}$ is well-defined. Indeed, if there exist $x_{1}, x_{2} \in \mathcal{A}_{1}$ such that $\pi(x_{1}) = \pi(x_{2})$, then $x_{1} - x_{2} \in Ker(\pi)$ and $f(x_{1} - x_{2}) \in Ker(\pi)$. So $\pi(f(x_{1} - x_{2})) = 0$, i.e., $\pi(f(x_{1})) = \pi(f(x_{2}))$. It's obvious that $\pi \circ f = \bar{f} \circ \pi$.

Suppose that there exists $\bar{f}^{'} \in \rm{End}_{\rm{F}}(\mathcal{A}_{2})$ such that $\pi \circ f = \bar{f}^{'} \circ \pi$. Then for any $y \in \mathcal{A}_{2}$, there exists $x \in \mathcal{A}_{1}$ such that $y = \pi(x)$. Moreover, $\bar{f}^{'}(y) = \bar{f}^{'}(\pi(x)) = \pi(f(x)) = \bar{f}(\pi(x)) = \bar{f}(y)$. So $\bar{f}^{'} = \bar{f}$. The uniqueness holds.

(1). For all $f, g \in \rm{End}_{\rm{F}}(\mathcal{A}_{1}, \it{Ker}(\pi))$, we have
\[\pi \circ (f \circ g) = (\pi \circ f) \circ g = (\bar{f} \circ \pi) \circ g = \bar{f} \circ (\pi \circ g) = \bar{f} \circ (\bar{g} \circ \pi) = (\bar{f} \circ \bar{g}) \circ \pi,\]
which implies that $\pi_{\rm{End}}\it{(f \circ g)} = \bar{f} \circ \bar{g} = \pi_{\rm{End}}(\it{f}) \circ \pi_{\rm{End}}(\it{g})$. So $\pi_{\rm{End}}$ is an algebra homomorphism.

(a). For any $R_{(x, y)} \in Mult(\mathcal{A}_{1})$ and for any $z \in Ker(\pi)$, we have
\[\pi(R_{(x, y)}(z)) = \pi(\rm{[\![\it{z},\it{x},\it{y}]\!]}) = \rm{[\![\pi(\it{z}),\pi(\it{x}),\pi(\it{y})]\!]} = 0,\]
which implies that $R_{(x, y)}(z) \in Ker(\pi)$, i.e., $R_{(x, y)} \in \rm{End}_{\rm{F}}(\mathcal{A}_{1}, \it{Ker}({\pi}))$. Hence, $Mult(\mathcal{A}_{1}) \subseteq \rm{End}_{\rm{F}}(\mathcal{A}_{1}, \it{Ker}({\pi}))$. Moreover, we see that $\pi \circ R_{(x, y)} = R_{(\pi(x), \pi(y))} \circ \pi$, so $\pi_{\rm{End}}(R_{\it{(x, y)}}) = R_{(\pi(x), \pi(y))}$. Moreover, $\pi$ is an epimorphism, so $\pi_{\rm{End}}(Mult(\mathcal{A}_{1})) = Mult(\mathcal{A}_{2})$.

For any $f \in \Gamma(\mathcal{A}_{1}) \cap \rm{End}_{\rm{F}}(\mathcal{A}_{1}, \it{Ker}({\pi}))$ and $x, y, z \in \mathcal{A}_{2}$, there exist $u, v, w \in \mathcal{A}_{1}$ such that $x = \pi(u)$, $y = \pi(v)$ and $z = \pi(w)$. Moreover, we have
\begin{align*}
&\bar{\it{f}}(\rm{[\![\it{x},\it{y},\it{z}]\!]}) = \bar{\it{f}}(\rm{[\![\pi(\it{u}),\pi(\it{v}),\pi(\it{w})]\!]}) = \bar{\it{f}} \circ \pi(\rm{[\![\it{u},\it{v},\it{w}]\!]}) = \pi \circ \it{f}(\rm{[\![\it{u},\it{v},\it{w}]\!]})\\
&= \pi(\rm{[\![\it{f}(\it{u}),\it{v},\it{w}]\!]}) = \rm{[\![\pi \circ \it{f}(\it{u}),\pi(\it{v}),\pi(\it{w})]\!]} = \rm{[\![\bar{\it{f}} \circ \pi(\it{u}),\pi(\it{v}),\pi(\it{w})]\!]} = \rm{[\![\bar{\it{f}}(\it{x}),\it{y},\it{z}]\!]},
\end{align*}
similarly, we have $\bar{\it{f}}(\rm{[\![\it{x},\it{y},\it{z}]\!]}) = \rm{[\![\it{x},\bar{\it{f}}(\it{y}),\it{z}]\!]} = \rm{[\![\it{x},\it{y},\bar{\it{f}}(\it{z})]\!]}$. So $\bar{f} \in \Gamma(\mathcal{A}_{2})$, i.e., $\pi_{\rm{End}}(\Gamma(\mathcal{A}_{1}) \cap \rm{End}_{\rm{F}}(\mathcal{A}_{1}, \it{Ker}({\pi}))) \subseteq \Gamma(\mathcal{A}_{2})$.

(b). $\forall f, g \in \Gamma(\mathcal{A}_{1}) \cap \rm{End}_{\rm{F}}(\mathcal{A}_{1}, \it{Ker}(\pi))$, we have
\[\pi \circ (f \circ g) = (\pi \circ f) \circ g = (\bar{f} \circ \pi) \circ g = \bar{f} \circ (\pi \circ g) = \bar{f} \circ (\bar{g} \circ \pi) = (\bar{f} \circ \bar{g}) \circ \pi,\]
which implies that $\pi_{\Gamma}(f \circ g) = \bar{f} \circ \bar{g} = \pi_{\Gamma}(f) \circ \pi_{\Gamma}(g)$. So $\pi_{\Gamma}$ is a homomorphism.

(c). For any $x \in Ker(\pi)$, $y, z \in \mathcal{A}_{1}$ and $\psi \in \Gamma(\mathcal{A}_{1})$, we have $\rm{[\![\psi(\it{x}),\it{y},\it{z}]\!]} = \psi(\rm{[\![\it{x},\it{y},\it{z}]\!]}) = 0$, which implies that $\psi(x) \in Z(\mathcal{A}_{1}) = Ker(\pi)$. So $\psi$ leaves $Ker(\pi)$ invariant.

(2). If $\bar{\varphi} = 0$ for $\varphi \in \Gamma(\mathcal{A}_{1}) \cap \rm{End}_{\rm{F}}(\mathcal{A}_{1}, \it{Ker}(\pi))$, then $\pi(\varphi(\mathcal{A}_{1})) = \bar{\varphi}(\pi(\mathcal{A}_{1})) = 0$. Hence, $\varphi(\mathcal{A}_{1}) \subseteq Ker(\pi) \subseteq Z(\mathcal{A}_{1})$.

Hence, $\varphi(\rm{[\![\it{x},\it{y},\it{z}]\!]}) = \rm{[\![\varphi(\it{x}),\it{y},\it{z}]\!]} = 0$. Note that $\mathcal{A}_{1}$ is perfect, we have $\varphi = 0$. Therefore, $\pi_{\Gamma}$ is injective.

(3). $\forall y, z \in \mathcal{A}_{2}, \exists y^{'}, z^{'} \in \mathcal{A}_{1}$ such that $y = \pi(y^{'})$ and $z = \pi(z^{'})$. For all $x \in Z(\mathcal{A}_{1})$,
\[\rm{[\![\pi(\it{x}),\it{y},\it{z}]\!]} = \rm{[\![\pi(\it{x}),\pi(\it{y}^{'}),\pi(\it{z}^{'})]\!]} = \pi(\rm{[\![\it{x},\it{y}^{'},\it{z}^{'}]\!]}) = 0,\]
which implies that $\pi(x) \in Z(\mathcal{A}_{2})$. So $\pi(Z(\mathcal{A}_{1})) \subseteq Z(\mathcal{A}_{2}) = \{0\}$. Therefore, $Z(\mathcal{A}_{1}) \subseteq Ker(\pi)$. Note that $Ker(\pi) \subseteq Z(\mathcal{A}_{1})$, we have $Z(\mathcal{A}_{1}) = Ker(\pi)$.

For all $\varphi \in \Gamma(\mathcal{A}_{1}), x \in Z(\mathcal{A}_{1}), y, z \in \mathcal{A}_{1}$,
\[\rm{[\![\varphi(\it{x}),\it{y},\it{z}]\!]} = \varphi(\rm{[\![\it{x},\it{y},\it{z}]\!]}) = 0,\]
which implies that $\varphi(x) \in Z(\mathcal{A}_{1}) = Ker(\pi)$. So $\varphi(Ker(\pi)) \subseteq Ker(\pi)$, i.e., $\varphi \in \rm{End}_{\rm{F}}(\mathcal{A}_{1}, \it{Ker}(\pi))$. So $\Gamma(\mathcal{A}_{1}) \subseteq \rm{End}_{\rm{F}}(\mathcal{A}_{1}, \it{Ker}(\pi))$. Therefore, $\Gamma(\mathcal{A}_{1}) \cap \rm{End}_{\rm{F}}(\mathcal{A}_{1}, \it{Ker}({\pi})) = \rm{\Gamma}(\mathcal{A}_{1})$. According to (1) (b) and (2), we have $\pi_{\Gamma}$ is a monomorphism.
\end{proof}

\end{document}